\documentclass[reqno,11pt]{amsart}
\usepackage{amssymb,amsmath,amsthm,amsfonts}
\usepackage{mathrsfs,dsfont, comment,mathscinet}
\usepackage{enumerate,esint}
\usepackage[left=2.8cm,right=2.8cm,top=2.8cm,bottom=2.8cm]{geometry}
\usepackage{mathtools}
\usepackage[pdftex]{color, graphicx}

\usepackage{tikz}
\usepackage{float} 
\usepackage{tkz-euclide}
\usetikzlibrary{arrows.meta}

\usepackage{mathtools} 

\usepackage{epstopdf}
\usepackage[linktocpage=true,colorlinks=true,linkcolor=blue,citecolor=green]{hyperref}

\allowdisplaybreaks
\numberwithin{equation}{section}
\theoremstyle{plain}

\usepackage[english]{babel}
\usepackage{amsthm}
\usepackage{amsfonts, mathrsfs}
\usepackage{latexsym}
\usepackage{comment} 
\usepackage{esint} 
\usepackage{pdfsync}
\usepackage{enumitem}

\usepackage{caption}    

\usepackage{cancel}     
\usepackage{ulem}       
\usepackage[]{todonotes}
\newtheorem{thm}{Theorem}[section]
\newtheorem{prop}[thm]{Proposition}
\newtheorem{lemma}[thm]{Lemma}

\newtheorem{remark}[thm]{Remark}

\theoremstyle{definition}

\allowdisplaybreaks[2] 		

\numberwithin{equation}{section}    

\newcommand{\e}{\varepsilon}
\newcommand{\ealpha}{\varepsilon^\alpha}
\newcommand{\loc}{\textnormal{loc}}
\newcommand{\dist}{\textnormal{dist}}
\newcommand{\supp}{\textnormal{supp}}
\renewcommand{\hom}{\textnormal{hom}}

\newcommand{\per}{\textnormal{per}}

\newcommand\wts{\overset{\text{w2}}{\rightharpoonup}}
\newcommand\sts{\overset{\text{s2}}{\to}}

\newcommand{\R}{\mathbb{R}}
\newcommand{\s}{\mathbb{S}}
\newcommand{\Y}{\mathcal{Y}}
\newcommand{\A}{\mathcal{A}}

\def\div{{\operatorname{div}}}
\newcommand{\dx}{\,\mathrm{d}x}
\newcommand{\dif}{\,\mathrm{d}}
\newcommand{\Id}{\,\mathrm{Id}}	
\def\esssup{\mbox{ess-sup}}

\newenvironment{samuelerev}{\color{purple}}{\color{black}}
\newcommand{\bsamr}{\begin{samuelerev}}
\newcommand{\esamr}{\end{samuelerev}}

\title[Homogenization and linearization in magnetoelasticity]{Homogenization and linearization in magnetoelasticity under small elastic response}

\author[M.\,Cherdantsev]{Mikhail Cherdantsev}
\address[M.\,Cherdantsev]{Cardiff University, Mathematical Analysis Research Group, Senghennydd Road, Cathays, CF24 4AG Cardiff, UK}
\email{CherdantsevM@cardiff.ac.uk}

\author[E.\,Davoli]{Elisa Davoli}
\address[E.\,Davoli]{Tu Wien, Institute for Analysis and Scientific Computing, Wiedner Hauptstra\ss{}e 8-10, 1040 Wien, Austria}
\email{elisa.davoli@tuwien.ac.at}

\author[L.\,D'Elia]{Lorenza D'Elia}
\address[L.\,D'Elia]{Tu Wien, Institute for Analysis and Scientific Computing, Wiedner Hauptstra\ss{}e 8-10, 1040 Wien, Austria}
\email{lorenza.delia@tuwien.ac.at}

\author[S.\,Riccò]{Samuele Riccò}
\address[S.\,Riccò]{Tu Wien, Institute for Analysis and Scientific Computing, Wiedner Hauptstra\ss{}e 8-10, 1040 Wien, Austria}
\email{samuele.ricco@tuwien.ac.at}



\begin{document}

\begin{abstract}
    We perform a simultaneous homogenization and linearization analysis for a magnetoelastic energy functional featuring a mixed Eulerian-Lagrangian structure. Neglecting Zeeman and anisotropic contributions, we characterize the asymptotic behavior in the sense of Gamma-convergence for the sum of a nonlinear magnetoelastic energy, a symmetric exchange term defined on the actual configuration, and for the associated magnetostatic self-energy. After establishing compactness of displacements and magnetizations with equibounded energy, we identify the limiting energy functional as the sum of a quadratic homogenized magnetoelastic contribution with a limiting homogenized exchange and magnetostatic term. This is, to the authors' knowledge, the first homogenization result for manifold-valued mixed Eulerian-Lagrangian energies.
\end{abstract}

\maketitle


\section{introduction}

In this work, we present a simultaneous linearization and homogenization analysis for magnetoelastic materials. This can be seen as a first step towards extending the seminal commutability results in the nonlinear elastic setting \cite{MuellerNeukamm, GloriaNeukamm} to the magnetoelastic regime.

The mathematical modeling of multiphysics couplings is at the center of a vigorous research effort, owing to the wide range of potential application, spanning, among others, sensors and actuators, as well as bioengineering devices. A crucial mathematical hurdle in describing such nonlinear couplings is the mixed Eulerian-Lagrangian contributions  in  the energy functionals, featuring energy densities defined both on reference configuration and on the actual one.

We consider a model for a periodic ferromagnetic elastic composite under a small elastic response regime. Starting from a fully non-linear elasticity setting combined with magnetization, we pass to the limit as both the size of the periodicity and the elastic energy tend to zero. This leads to simultaneous linearization in elastic properties and homogenization in elasticity and magnetization via the $\Gamma$-convergence technique.  

The variational theory of static magnetostriction (cf., e.g.,\cite{brown}) stems from the assumption that equilibria minimize an underlying energy functional, depending both on elastic deformations of the reference configuration $w:\Omega \to \Omega^{w} \subset \R^3$ and on the magnetization  $m:\Omega^{w}\to\mathbb{S}^2$, where $\mathbb{S}^2$ denotes the unit sphere in $\R^3$ and $\Omega^{w}$ is the deformed set. A simplified  energy functional in the case of a single-crystal magnet reads as
\begin{equation}
    \label{eqn:energy-E-0}
    G(w,m) := \int_\Omega W(\nabla w, m\circ w) \, \dx + a \int_{\Omega^w}|\nabla m|^2\,{\rm d}\xi + \frac{\mu_0}{2} \int_{\R^3}|\nabla \psi_m|^2\,{\rm d}\xi. 
\end{equation}
Here, $W$ encodes a nonlinear, frame-indifferent, magnetostrictive energy density. The second term in \eqref{eqn:energy-E-0} describes the magnetic exchange, penalizing the spatial changes of $m$, while  the last term  encodes the contributions of the magnetic stray field $\psi_m$. The parameters $a>0$ and $\mu_0>0$ are the exchange constant and the permeability of the vacuum, respectively. The stray field potential $\psi_m$ is the weak solution of the Maxwell equation 
\begin{equation}
\label{maxwell}
    \div(\chi_{w(\Omega)} m ) = \mu_0 \Delta \psi_m \quad \textnormal{ in } \mathcal{D}'(\R^3),
\end{equation}
where $\chi_{w(\Omega)} m$ denotes the extension of $m$ by $0$ to the whole space. In particular, $\psi_m$ can be seen as an element of the Beppo-Levi space $L^{1,2}(\R^3)$, see \eqref{eq:BeppoLevi} below. The existence of this weak solution follows from \cite[Proposition 8.8]{BHMC17}, see also \cite[Lemma 4.2]{B21}. For ease of presentation, in \eqref{eqn:energy-E-0} and throughout this paper we disregard the energy contributions due to crystalline anisotropy, antisymmetric exchange and Zeeman energy, since they behave as continuous perturbations and could easily be included in our analysis the  without further mathematical difficulties. 

Well-posedness of the minimization problem for the functional in \eqref{eqn:energy-E-0} in suitable classes of admissible deformations and magnetizations has first been established in \cite{rybka.luskin} under a higher-order regularization corresponding to the setting of non-simple materials. Existence results in the absence of hessian terms have first been proven in \cite{kruzik.stefanelli.zeman}, assuming incompressibility of the admissible deformations, as well as in \cite{BHMC17} under weaker coercivity conditions on the energy density and weaker notions of deformed sets. Further improvements have been achieved in \cite{BrescianiDavoliKruzik, Bresciani, BrescianiThesis}. We also mention \cite{BrescianiFriedrichMora-Corral} and \cite{BrescianiFriedrich} for a general study of Eulerian-Lagrangean models accounting for material failure, as well as for a characterization of the emergence of domain structures. A linearization analysis has been undertaken in \cite{AKM24}. A homogenization analysis in a linear setting under the presence of rigid magnetic inclusions has been carried out in the recent preprint \cite{GrandeKromerKruzikTomassetti}.

In order to set the problem we first specify some  notation. Let $\Omega \subset \R^3$ be a bounded Lipschitz domain in $\R^3$, and let $\gamma$ be an open subset of $\partial \Omega$ such that $\mathcal{H}^2(\gamma) > 0$, where $\mathcal{H}^2$ denotes the $2$-dimensional Hausdorff measure. From now on we assume $\Omega$ to be the reference configuration of a composite magnetoelastic material. Let also $p\in \mathbb{N}$ be such that $p>3$, and let $Y := [0,1)^3$ be the unit cube in $\R^3$ which, with a slight abuse of notation, we implicitly identify with the unit torus. We denote by $g$  a  function in $W^{1,\infty}(\Omega;\R^3)$ encoding the boundary conditions on $\gamma$. Further, we  denote by $\{e_i\}_{i=1}^3$  the canonical basis of $\R^3$. Given a vector $u \in \R^3$ we indicate its components by $u = (u^1, u^2, u^3)$, and adopt the symbol $(\cdot|\cdot|\cdot)$ to describe the columns of a matrix. We will use a standard notation for the Beppo-Levi space 
\begin{equation}
\label{eq:BeppoLevi}
    L^{1,2}(\R^3) := \{u\in L^2_\loc(\R^3) \ : \ \nabla u \in L^2(\R^3; \R^3) \}.
\end{equation}
We recall that this is an Hilbert space, modulo constant functions (see \cite[Corollary 1.1]{DL54}), equipped with the inner product
\begin{equation*}
   (u,v) := \int_{\R^3} \nabla u \cdot \nabla v \dx.
\end{equation*}
For a given sequence $(v_\e)_\e$ and for $r\in \mathbb{N}$, we will write $v_\e \wts v$ weakly in $L^r(\Omega \times Y)$ to indicate that $(v_\e)_\e$ converges weakly two-scale in $L^r(\Omega\times Y)$ to $v$. Analogously, we will use the notation $v_\e \sts v$ strongly in $L^r(\Omega\times Y)$ to indicate its strong two-scale convergence. We will often indicate by $C$ a generic positive constant, whose value may change from line to line.

Following \cite{AD15} we work under the assumption of the strong coupling condition, ensuring that the admissible magnetizations do not jump across the interface between different crystal magnets. Assuming a periodicity scale $\e$, our stored energy is thus described by the functional
\begin{equation}
\label{eq:en-hom}
\begin{aligned}
    G_\e(w,m)
    :=& \int_\Omega W \left(\frac{x}{\e}, w(x), m \circ w(x) \right) \dx \\
    +& \int_{\Omega^w} a\left(\frac{w^{-1}(z)}{\e}\right) |\nabla m(z)|^2 \dif z + \frac{\mu_0}{2} \int_{\R^3} |\nabla \psi_m|^2 \dx,
\end{aligned}
\end{equation}
for pairs of admissible deformations and magnetizations $(w,m)\in \mathcal{A}$, and with $G_\e(w,m)=+\infty$ otherwise. In the expression above, $W : \R^3 \times \R^{3 \times 3} \times \R^3 \to [0,+\infty)$, is a Carath\'eodory function satisfying the following assumptions:

\begin{enumerate}[label= {\rm  (H\arabic*)}]
    \item\label{H1} {\bf Periodicity:} $W(\cdot, F, \nu)$ is $Y$-periodic for any $F \in \R^{3 \times 3}$ and for any $\nu \in \R^3$.
    
    \item\label{H2} {\bf Coercivity:} there exists a constant $C > 0$ such that 
    \begin{equation*}
        W(y, F, \nu) \ge C \, \left[\dist^2(F; SO(3)) \vee \dist^p(F; SO(3)) \right],
    \end{equation*}
    for some $p > 3$ and for any $F \in \R^{3 \times 3}$, $\nu \in \R^3$ and a.e. $y \in \R^3$. Moreover we assume
    \begin{equation*}
        W(y, F, \nu) \ge h(\det F)
    \end{equation*}
    for any $F \in \R^{3 \times 3}$, $\nu \in \R^3$ and a.e. $y \in \R^3$, where $h : (0, +\infty) \to [0, +\infty)$ is a Borel function satisfying
    \begin{equation}
    \label{grwt-h}
        \lim_{t \to 0^+} t h(t) = +\infty.
    \end{equation}
    In particular, we choose $h(t) = t^{-s}$ where $s > \frac{p}{p-2}$.
    
    \item\label{H3} {\bf Frame indifference:} $W(y, RF, R\nu) = W(y, F, \nu)$ for any $R \in SO(3)$, for any $F \in \R^{3 \times 3}$, for any $\nu \in \R^3$ and for a.e. $y \in \R^3$.
    
    \item\label{H4} {\bf Normalization:} $W(y, \Id, \nu) = 0$ for any $\nu \in \R^3$ and for a.e. $y \in \R^3$.
    
    \item\label{H5} {\bf Taylor expansion:} $W$ has a second--order Taylor expansion centered at the identity, namely
    \begin{equation*}
        W(y, \Id + F, \nu) = \frac{1}{2} \, Q(y, F, \nu) + r(y, F, \nu),
    \end{equation*}
    for any $F \in \R^{3 \times 3}$, for any $\nu \in \R^3$ and for a.e. $y \in \R^3$. Here $Q: \R^3 \times \R^{3 \times 3} \times \R^3 \to [0, +\infty)$ is a non-negative Carathéodory function which is $Y$-periodic in $y$ and quadratic in $F$. Moreover, we require that $Q$ is bounded in the following sense: for any $F \in \R^{3 \times 3}$ and for any $\nu \in \R^3$,
        \begin{equation}
            \notag
            \underset{y \in \R^3}{\esssup} \ Q(y, F, \nu) \le C \, |F|^2,
        \end{equation}
    for some constant $C > 0$. We further require that for any $\nu \in \R^3$ and for a.e. $y \in \R^3$, $r(y, F, \nu) \le \tilde{r}(|F|)|F|^2$ as $|F| \to 0$, for a suitable continuous function $\tilde{r}$.
\end{enumerate}

We observe that condition \eqref{grwt-h} is a strengthening of the assumption that
\begin{equation*}
    \lim_{t \to 0^+} h(t) = +\infty,
\end{equation*}
which is fundamental from the modeling point of view, for it ensures that the elastic energy blows up in the case of extreme compressions. 
%
%

The map $a : \R^3 \to \R$ is a $Y$-periodic function such that $a \in L^\infty(Y)$, and there exist $C_1, C_2 > 0$ for which 
\begin{equation}
\label{grw-a}
    0 < C_1 \le a(y) \le C_2 \quad \textnormal{ for a.e. }  y \in \R^3. 
\end{equation}
We define the class $\mathcal{A}$ of admissible deformations and magnetizations as
\begin{equation*}
\begin{aligned}
    \mathcal{A} := \bigg\{&(w,m)\in W^{1,p}_g(\Omega;\R^3)\times H^1(\Omega^w;\mathbb{S}^2) \ : \ \det \nabla w > 0 \textnormal{ a.e. in } \Omega,
    \\
    & w \textnormal{ is a.e. injective in } \Omega, \ |\det\nabla w|^{-1} \in L^s(\Omega) \ \textnormal{ for } s > \frac{p}{p-2}\bigg\},
\end{aligned}
\end{equation*}
where
\begin{equation*}
    W^{1,p}_g(\Omega; \R^3) := \{ v \in W^{1,p}(\Omega; \R^3) \ : \ v = g \quad \mathcal{H}^2\textnormal{-a.e. on } \gamma \}.
\end{equation*}
Note that we work under the assumption $p>3$ in order for the admissible deformations to have a continuous representative, {\it i.e.} for the deformed configuration to be well-defined point-wise. 

In the definition above and in \eqref{eq:en-hom}, the deformed set $\Omega^w$ is defined as $\Omega^w := w(\Omega) \setminus w(\partial \Omega)$. Note that, in principle, even for deformations $w$ in our class, the set $w(\Omega)$ might not be open, whereas $\Omega^w$ is always an open set, cf. \cite[Lemma 2.1]{BrescianiDavoliKruzik}. We will use the notation
\begin{equation}
\label{eq:class-Y}
\begin{aligned}
    \Y := \bigg\{ &w\in W^{1,p}_g(\Omega; \R^3) \ : \ \det \nabla w > 0 \textnormal{ a.e. in } \Omega, \ w \textnormal{ is a.e. injective in } \Omega, 
    \\
    & |\det\nabla w|^{-1} \in L^s(\Omega) \quad \textnormal{ for some } s > \frac{p}{p-2} \bigg\}.
\end{aligned}
\end{equation}
 for the class of admissible deformations only.


In the absence of a deformation, assuming that the temperature of the sample is fixed below the Curie threshold, the magnetization satisfies the Heisenberg constraint $|m| = 1$. In the presence of a deformation it is natural to assume that $|m \circ w| \det \nabla w = 1$, where the Jacobian takes care of the change of the volume of an elementary sample. This, however introduces considerable technical difficulties in terms of the regularity of the involved terms and when establishing compactness properties of the magneto-elastic states. In our setting, when the deformations are close to identity, we instead make the assumption   
\begin{equation}
\label{Heisenberg}
    |m \circ w| = 1 \quad \textnormal{a.e. in } \Omega,
\end{equation}
dropping the term $\det \nabla w$, which is reasonable from the modeling point of view and allows us to avoid unnecessary technical complications. 
 In the fully nonlinear setting, different formulations of the Heisenberg constraint have been proposed in the literature, cf. \cite{RT15} for a related discussion. 


The focus of our work is on a simultaneous homogenization and linearization analysis. To this end we introduce a scaling factor $\frac{1}{\e^{2\alpha}}$, for some $\alpha > 0$, in front of the magnetostrictive energy term in \eqref{resc-en} below. An application of the geometric rigidity result \cite{FJM02} then yields that  our admissible deformations are close to the identity. Therefore, for every deformation $w \in \Y$ and for any $\e > 0$ we introduce the rescaled displacement
\begin{equation*}
    u(x) := \frac{w(x) - x}{\ealpha},
\end{equation*}
The parameter $\alpha$ can be considered, roughly speaking, describe the ``relative speed''   of the linearization and homogenization processes. In the full 3d setting (at list in our case), however, in contrast  to the dimension reduction problems, e.g. plate theories, the analysis and the limit problem do not depend on $\alpha$.  

Our rescaled magnetoelastic energy $F_\e$ reads as  
\begin{equation}
\label{resc-en}
\begin{aligned}
    F_{\e}(u,m)
    := \ &\frac{1}{\e^{2\alpha}} \int_\Omega W \left(\frac{x}{\e}, \Id + \e^\alpha \nabla u(x), m(x + \e^\alpha u(x)) \right) \dx \\
    + &\int_\Omega a \left(\frac{x}{\e}\right) |\nabla m(x + \e^\alpha u(x))|^2 |\det(\Id+\e^\alpha\nabla u)| \dx + \frac{\mu_0}{2} \int_{\R^3} |\nabla \psi_m|^2 \dx,
\end{aligned}
\end{equation}
for all $(u,m)\in \mathcal{A}_{\e}$, where the class $\mathcal{A}_{\e}$ of admissible states is defined as 
\begin{equation}
\label{adm}
    \A_\e := \{ (u, m) \in W^{1,p}_g(\Omega; \R^3) \times H^1((id + \e^\alpha u)(\Omega); \mathbb{S}^2) \ : \ id + \e^\alpha u \in \Y \}.
\end{equation}
We stress that, for $\e$ small enough and $u \in W^{1,p}_g(\Omega;\R^3)$, the map $id + \e^\alpha u$ is a homeomorphism, so that the deformed set $(id + \e^\alpha u)(\Omega)$ in the expression above is open and coincides, up to a set of measure zero, with $\Omega^{(id + \e^\alpha u)(\Omega)}$. (Note the difference in the way the boundary condition  is formulated in $\A_\e$ compared to $\A$.)

The main result of this work is a characterization of the asymptotic behaviour of the functionals $F_\e$ as $\e \to 0$. The limiting functional $F_\hom$ is given by
\begin{equation}
\label{hommagnetoel}
\begin{aligned}
    F_{\hom}(u, m)
    := & \ \frac{1}{2} \int_\Omega Q_\hom(\nabla u(x), m(x)) \dx + \int_\Omega T_\hom(\nabla m(x)) \dx \\
    + & \ \frac{\mu_0}{2} \int_{\Omega}| \nabla\psi_m(x)|^2 \dx,
\end{aligned}
\end{equation}
for every $(u,m)\in \mathcal{A}_0$, where
\begin{equation*}
    \A_0 := W^{1,p}_g(\Omega; \R^3) \times H^1(\Omega; \s^2).
\end{equation*}
In \eqref{hommagnetoel}, the linearized and homogenized density $Q_\hom$ is a quadratic form given by 
\begin{equation}
\label{Qhom}
    Q_\hom(A, \nu) := \inf \left\{ \int_Y Q(y, A + \nabla \phi(y), \nu) \dif y \ : \ \phi \in H^1_\per(Y;\R^3) \right\},
\end{equation}
for any $A \in \R^{3 \times 3}$ and $\nu \in \s^2$, where $Q$ is  defined in \ref{H5}. In turn, $T_\hom$ is the homogenized exchange energy density define as
\begin{equation}
\label{Thom}
    T_\hom(A) := \inf \left\{ \int_Y a(y) \, |A + \nabla \phi(y)|^2 \dif y \ : \ \phi \in H^1_\per(Y; \R^3) \right\},
\end{equation}
for any $A \in \cup_{s\in \s^2}T_s \, \s^2$, where $T_s \, \s^2$ is the tangent space of $\s^2$ at the point $s \in \s^2$.

One may expect that the minimization in \eqref{Thom} is taken over the tangent bundle. This is indeed the case, however, due to the sphere being invariant under rotations, it is possible to reduce the problem  to an unconstrained minimization, see \cite[Section 3(b) and Remark 3.3]{AD15}.
 In the general case, where manifold-valued maps are taken under consideration, an explicit dependence on the point-wise tangent space must also be taken into account. 

We are now in a position to state our main result.

\begin{thm}
\label{main_thm}
    Let $\alpha>0$, and assume \ref{H1}-\ref{H5}. Then, the family of magnetoelastic energies $(F_\e)_\e$, given by \eqref{resc-en}, $\Gamma$-converges with respect to the convergence in Propositions \ref{compact} and \ref{prop:two_scale} to the homogenized functional $F_{\hom}$ defined by \eqref{hommagnetoel}.
    \qed
\end{thm}

The proof of Theorem \ref{main_thm} hinges upon several intermediate results. The key hurdles are understanding the compactness properties of a sequence of states $(u_\e,m_\e)\in \mathcal{A}_\e$ with equibounded energies, and a characterization of the two-scale limit of the magnetization. They are the subjects of Propositions \ref{compact} and \ref{prop:two_scale}, respectively. A crucial difficulty for both results is the fact that magnetization is an {\it actual} quantity and, thus, is defined and equibounded only on the deformed set. This requires a careful simultaneous analysis of the displacement and the magnetization, as well as a stability result for the inverse of the determinant of the admissible deformations. This is the main motivation for the choice of $h$ in assumption \ref{H2}, for it implies that for a bounded energy sequence not only satisfies
\begin{equation}
\label{5.15a}
    |\det\nabla w|^{-1} \in L^s(\Omega) \quad \textnormal{ for } \ s > \frac{p}{p-2}
\end{equation}
for any $w$ in the sequence, but also implies that $|\det\nabla w_\e|^{-1}$ is uniformly bounded in $L^s(\Omega)$.
\\
Our second contribution is to show in Section \ref{sec:commute} that the functional $F_{\hom}$ in \eqref{hommagnetoel} can be recovered by first linearizing magnetostrictive contribution and then passing to the  homogenization limit. This is illustrated in the following diagram, where we denote by $G_{\e}^{\textrm{lin}}$ the corresponding linearized functional.

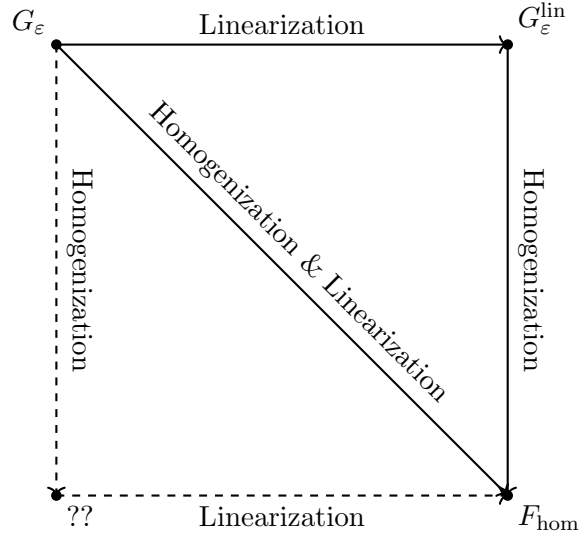
\begin{figure}[H]
\begin{center}
\begin{tikzpicture}
    \tkzDefPoint(0,6){A}   
    \tkzDefPoint(6,6){B}   
    \tkzDefPoint(0,0){D}   
    \tkzDefPoint(6,0){C}   
    \tkzDefPoint(6,0.1){K}
    \tkzDefPoint(5.9,0.1){Z}
    
    \fill (0,6) circle (2pt) node[above left] {$G_{\e}$};
    \fill (6,6) circle (2pt) node[above right] {$G_{\e}^{\textrm{lin}}$};
    \fill (6,0) circle (2pt) node[below right] {$F_{\textrm{hom}}$};
    \fill (0,0) circle (2pt) node[below right] {$??$};
    
    \draw[->, thick] (A) -- (B) node[midway, above] {Linearization};
    \draw[->, thick] (B) -- (C)
    node[midway, above, sloped] {Homogenization};
    \draw[->, thick] (A) -- (C)
    node[midway, above, sloped] {Homogenization \& Linearization};
    \draw[->, dashed, thick] (A) -- (D) node[midway, above,sloped] {Homogenization};
    \draw[->, dashed, thick] (D) -- (C) node[midway, below] {Linearization};
\end{tikzpicture}
\end{center}
\end{figure}

The study of the missing dashed lines, on the other hand, remains instead an open problem, for, even in the setting of nonlinear elasticity (without magnetic contribution), homogenization under invertibility constraints is currently not fully understood.

The paper is organized as follows. In Section \ref{sec:auxiliary} we collect necessary preliminary results. The precompactness of sequences with equibounded energy is the focus of Section \ref{sec-compact}. A separate two-scale analysis of the asymptotic behavior of the magnetizations is carried out in Section \ref{sec_2sc}. The proof of Theorem \ref{main_thm} is presented in Section \ref{gamma-conv}. Finally, the sequential linearization and homogenization is discussed in Section \ref{sec:commute}.




\section{Preliminary results}
\label{sec:auxiliary}

To establish compactness, we rely upon the seminal geometric rigidity result in \cite{FJM02, FJM02-p}. We refer to \cite[Section 2.4]{CS06} for the case $p \neq 2$, as well as to \cite[Remark 2.6]{B21} for the second part of the statement.

\begin{thm}
\label{thm_rigest}
    Let $p \in (1, +\infty)$. Let $\Omega \subset \R^3$ be a bounded Lipschitz domain. Then, there exists a constant $C = C(\Omega, p) > 0$ such that for every $v \in W^{1,p}(\Omega; \R^3)$ there exists a constant rotation $R \in SO(3)$ satisfying
    \begin{equation*}
        \int_\Omega |\nabla v - R|^p \, \dx \le C \int_\Omega \dist^p(\nabla v, SO(3)) \, \dx.
    \end{equation*}
    Further, the constant rotation $R \in SO(3)$ is independent of the exponent $p$. In other words,
     \begin{equation*}
        \int_\Omega |\nabla v - R|^r \, \dx \le C \int_\Omega \dist^r(\nabla v, SO(3)) \, \dx.
    \end{equation*}
    for every $1 < r\le p$.
    \qed
\end{thm}

The following lemma is also instrumental for our compactness analysis.

\begin{lemma}
\label{l5.5}
    Let $\Omega \subset \R^3$ be a bounded domain and let $(f_\e)_\e\subset L^1(\Omega)$ be such that $f_\e \to 0$ strongly in $L^1(\Omega)$. If $(f_\e)_\e$ is bounded in $L^{r}(\Omega)$ for some $r>1$, then $f_\e \to 0$ strongly in $L^{r^*}(\Omega)$ for any $r^*\in [1,r)$.
\end{lemma}
\begin{proof}
    Note that $(f_\e)_\e$ is bounded in $L^{r^\ast}$ for every $r^{\ast} \in [1, r)$. Let us now fix $r^{\ast} \in [1, r)$ and let $\theta\in (0,1)$ be such that $\tfrac{1}{r^\ast} = \tfrac{1-\theta}{r} + \theta$. Since $(|f_\e|^{r^\ast(1-\theta)})_\e$ is bounded in $L^{\frac{r}{r^\ast(1-\theta)}}(\Omega)$ and $(|f_\e|^{r^\ast\theta})$ converges to zero strongly in $L^{\frac{1}{r^\ast\theta}}(\Omega)$, the claim follows directly from H\"older's inequality.      
\end{proof}

In order to identify a lower bound for the asymptotic behavior of the magnetoelastic energy contributions, we work with the following variant of \cite[Lemma 2.3.2]{N10}. For the basic properties of two-scale convergence, we refer the interested reader to \cite{Ng89, A92}.

\begin{lemma}
\label{lmm-neuk}
    Let $\Omega \subset \R^3$ be a bounded domain and let $\alpha > 0$. Let $(v_\e)_\e$ be a bounded sequence in $L^2(\Omega; \R^3)$ such that $v_\e \wts v$ weakly  in $L^2(\Omega \times Y; \R^3)$. If for every $\e > 0$ we define
    \begin{equation*}
        E_\e := \left\{ x \in \Omega \ : \ |v_\e(x)| \le \e^{-\frac{\alpha}{2}} \right\},
    \end{equation*}
    then
    \begin{equation*}
        \chi_{E_\e} v_\e \wts v \quad \textnormal{ weakly  in } L^2(\Omega \times Y; \R^3).
    \end{equation*}
\end{lemma}
\begin{proof}
    By definition we have
    \begin{equation*}
        \int_\Omega |v_\e(x)|^2 \dx \ge \e^{-\alpha} \int_\Omega (1 - \chi_{E_\e}(x)) \dx.
    \end{equation*}
    Since $|1 - \chi_{E_\e}|^r = 1 - \chi_{E_\e}$ for all $r \in [1,+\infty)$, we deduce that $1 - \chi_{E_\e}$ strongly converges to $0$ in $L^r(\Omega)$ for all $r \in [1,+\infty)$. Consequently, $\chi_{E_\e}$ strongly converges to $1$ in $L^r(\Omega)$ and, in particular, in measure. By applying \cite[Proposition 2.1.13]{N10} to the product $\chi_{E_\e} v_\e$, we infer the thesis.
\end{proof}

We further rely on the following result, whose proof can be found in \cite[Proposition 8.8]{BHMC17}.

\begin{lemma}
\label{lmm:sol_magneto_self}
    For every $v \in L^2(\R^3; \R^3)$, there exists a unique $\psi_v\in L^{1,2}(\R^3)$ (up to additive constants), such that
    \begin{equation*}
        \int_{\R^3} v \cdot \nabla \varphi \dx = \int_{\R^3} \nabla\psi_v \cdot \nabla \varphi \dx, \quad \forall \varphi\in L^{1,2}(\R^3),
    \end{equation*}
    and 
    \begin{equation*}
        \|\psi_v\|_{L^{1,2}(\R^3)} = \|\nabla \psi_v\|_{L^2(\R^3 \times \R^3)} \leq \|v\|_{L^2(\R^3)}.
    \end{equation*}
    \qed
\end{lemma}

Now we need to recall sufficient conditions for a deformation to be injective if the gradient of the displacement is small enough in the operator norm, namely $|A|_O = \sup_{\xi \in \R^d} \tfrac{|A\xi|}{|\xi|}$. The proof can be found in \cite[Theorem 5.5-1]{ciarlet}.

\begin{thm}
\label{thm:injective}
    Let $\Omega \subset \R^d$ an open set and let $w := id + u : \Omega \to \R^d$ be differentiable in $x \in \Omega$. If $|\nabla u(x)|_O < 1$, then $\det \nabla w(x) > 0$. Moreover, let $\Omega \subset \R^d$ be a bounded domain with Lipschitz boundary. Then there exists a constant $c(\Omega)$ such that any function $w := id + u \in C^1(\overline{\Omega};\R^d)$ with $\sup_{x \in \overline{\Omega}} |\nabla u(x)|_O < c(\Omega)$ is injective.
    \qed
\end{thm}


\section{Compactness}
\label{sec-compact}

In this section, we adapt the approach in \cite{AKM24} to characterize compactness properties for sequences of displacements and magnetizations with equibounded energies. Our main result is the following.

\begin{prop}
\label{compact}
    Let $\alpha > 0$, and assume \ref{H1}--\ref{H5}. Let $(u_\e, m_\e) \in \A_\e$ be such that 
    \begin{equation}
    \label{suplim}
        \sup_{\e > 0} F_\e(u_\e, m_\e) < + \infty,
    \end{equation}
    where the energy $F_\e$ is defined as in \eqref{resc-en}. Then, there exists $(u,m) \in \A_0$ such that, up to subsequences, setting $w_\e := id + \e^\alpha u_\e$, there holds
    \begin{alignat*}{3}
        w_\e &\to id &&\qquad \textnormal{strongly in } W^{1,p}(\Omega; \R^3), \\
        u_\e &\rightharpoonup u &&\qquad \textnormal{weakly in } W^{1,2}(\Omega; \R^3), \\
        \det\nabla w_\e &\to 1 &&\qquad \textnormal{strongly in } L^\frac{p}{3}(\Omega), \\ 
        |\det\nabla w_\e|^{-1} &\to 1 &&\qquad \textnormal{strongly in } L^1(\Omega), \\     
        m_\e \circ w_\e &\rightharpoonup m  &&\qquad \textnormal{weakly in } W^{1,q}(\Omega; \R^3), \\
        m_\e \circ w_\e &\to m  &&\qquad \textnormal{strongly in } L^r(\Omega; \R^3) \quad \textnormal{for any } r \in [1,+\infty), \\
        \chi_{w_\e(\Omega)} m_\e &\to \chi_\Omega m &&\qquad \textnormal{strongly in } L^r(\R^3; \R^3) \quad \textnormal{for any } r \in [1,+\infty), \\
        \chi_{w_\e(\Omega)} \nabla m_\e &\rightharpoonup \chi_\Omega \nabla m &&\qquad \textnormal{weakly in } L^2(\R^3; \R^{3 \times 3}),
    \end{alignat*}
    where $q := \frac{2ps}{p+s(p+2)} \in (1, \tfrac{2p}{p+2})$, and $s$ is the exponent in \eqref{5.15a}.
\end{prop}
\begin{proof}
    We subdivide the proof into several steps.
    \\
    {\bf{Step 1: Convergence of deformations.}} We first prove the following two inequalities:
    \begin{equation}
    \label{est_1}
        \int_\Omega |\nabla u_\e|^2 \dx \le C \left( F_\e(u_\e,m_\e) + \int_\gamma |g|^2 \dif \mathcal{H}^2 \right),
    \end{equation}
    and
    \begin{equation}
    \label{est_2}
        \int_\Omega |\e^\alpha \nabla u_\e|^p \dx \le C \e^{2\alpha} \left( F_\e(u_\e, m_\e) + \int_\gamma |g|^2 \dif \mathcal{H}^2 \right).
    \end{equation}
    Set
    \begin{equation*}
        g_p(t) := \max\{t^2, t^p\}, \quad \textnormal{for } t \in [0, +\infty).
    \end{equation*}
    Thanks to \ref{H2}, we have that
    \begin{equation}
    \label{rigest}
    \begin{aligned}
        \int_\Omega g_p(\dist(\Id + \e^\alpha \nabla u_\e; SO(3))) \dx
        \le & \, C \int_\Omega W\left(\frac{x}{\e}, \Id + \e^\alpha \nabla u_\e(x), m_\e(x + \e^\alpha u(x))\right) \dx \\
        \le & \, C \e^{2\alpha} F_\e(u_\e, m_\e).
    \end{aligned}
    \end{equation}
    In view of \eqref{rigest}, by Theorem \ref{thm_rigest} we find rotations $R_\e \in SO(3)$ such that
    \begin{equation*}
        \int_\Omega |\nabla w_\e - R_\e|^2 \dx \le C \int_\Omega g_p(\dist(\nabla w_\e, SO(3))) \dx \le C \e^{2\alpha} F_\e(u_\e, m_\e).
    \end{equation*}
    On the one hand, by the triangle inequality we obtain
    \begin{equation*}
        \e^{2\alpha} \int_\Omega |\nabla u_\e|^2 \dx = \int_\Omega |\nabla w_\e - \Id|^2 \dx \le C \left[ \e^{2 \alpha} F_\e(u_\e, m_\e) + |R_\e - \Id|^2 \right].
    \end{equation*}
    On the other hand, arguing as in \cite[Proposition 3.4]{DNP02}, we deduce
    \begin{equation}
    \label{rot-id}
        |R_\e - \Id|^2 \le C \e^{2\alpha} \left( F_\e(u_\e, m_\e) + \int_\gamma |g|^2 \dif \mathcal{H}^2 \right).
    \end{equation}
    Plugging the last two estimates together, we obtain  \eqref{est_1}, and, hence, the weak precompactness of the displacements $(u_\e)_\e$ in $W^{1,2}(\Omega;\mathbb{R}^3)$.
    \\
    From \eqref{rigest} and Theorem \ref{thm_rigest}, for the same rotation $R_\e$ as in \eqref{est_1}, there holds
    \begin{equation*}
        \int_\Omega |\nabla w_\e - R_\e|^p \dx \le C \int_\Omega g_p(\dist(\nabla w_\e, SO(3))) \dx \le C \e^{2\alpha} F_\e(u_\e, m_\e).
    \end{equation*}
    Hence, using the facts that $p > d = 3$ and that $|R_\e - \Id| \le 2\sqrt{d}$, since $R_\e \in SO(3)$, we obtain
    \begin{equation*}
    \begin{aligned}
        \int_\Omega |\e^\alpha \nabla u_\e|^p \dx
        &\le 2^{p-1} \left( \int_\Omega |\nabla w_\e - R_\e|^p \dx + |R_\e - \Id|^p \right) \\
        &\le C \left( \e^{2\alpha} F_\e(u_\e, m_\e) + |R_\e - \Id|^2 \right),
    \end{aligned}
    \end{equation*}
    which, together with \eqref{rot-id}, yields \eqref{est_2}.
    \\
    Now, \eqref{est_2} implies that
    \begin{equation}\label{eq:convrate}
        \| \nabla w_\e - \Id \|_{L^p(\Omega; \R^{3 \times 3})} \le C \e^\frac{2 \alpha}{p}
    \end{equation}
    for a positive constant $C$ independent of $\e$, which in turn implies that $w_\e \to id$ strongly in $W^{1,p}(\Omega; \R^3)$.
    
    \smallskip
    {\bf{Step 2: Convergence of the determinants.}} Recalling that the determinant is a multilinear function, we deduce that 
    \begin{equation*}
    \begin{aligned}
        \det(\nabla w_\e) &= \det(e_1+\e^\alpha\partial_1 u_\e|e_2+\e^\alpha\partial_2 u_\e|e_3+\e^\alpha\partial_3 u_\e)\notag\\
        &= 1+ \det(\e^\alpha \nabla u_\e)+\det(e_1|e_2|\e^\alpha\partial_3u_\e) +\det(e_1|\e^\alpha\partial_2 u_\e|e_3)+ \det(e_1|\e^\alpha\partial_2 u_\e|\e^\alpha\partial_3u_\e)\notag\\
        &\quad +\det(\e^\alpha\partial_1 u_\e|e_2|e_3)+\det(\e^\alpha\partial_1 u_\e|e_2|\e^\alpha\partial_3 u_\e) +\det(\e^\alpha\partial_1 u_\e|\e^\alpha\partial_2 u_\e|e_3)\notag\\
        &= 1+ \det(\e^\alpha \nabla u_\e)+ \e^\alpha\partial_3  u_\e^3+ \e^\alpha\partial_2 u_\e^2 +\e^{2\alpha}\partial_2 u_\e^2\partial_3 u_\e^3- \e^{2\alpha}\partial_3 u_\e^2\partial_2 u_\e^3 \notag\\
        &\quad + \e^\alpha\partial_1u_\e^1+ \e^{2\alpha}\partial_1u_\e^1\partial_3 u_\e^3- \e^{2\alpha}\partial_3 u_\e^1\partial_1 u_\e^3 + \e^{2\alpha}\partial_1 u_\e^1\partial_2 u_\e^2- \e^{2\alpha}\partial_2 u_\e^1\partial_1 u_\e^2.
    \end{aligned}
    \end{equation*}
    Thanks to the triangle inequality, for every $\sigma \in \mathbb{N}$ we find
    \begin{equation*}
    \begin{aligned}
        \int_\Omega |\det(\nabla w_\e)-1|^\sigma \dx \le C\left(\int_\Omega |\e^\alpha\nabla u_\e|^{3\sigma}\dx +  \int_\Omega |\e^\alpha\nabla u_\e|^\sigma \dx + \int_\Omega |\e^\alpha\nabla u_\e|^{2\sigma}\dx\right).
    \end{aligned}
    \end{equation*}
    Choosing $\sigma := \min\left\{p, \tfrac{p}{2}, \tfrac{p}{3}\right\}$, namely $\sigma: = \tfrac{p}{3}$, we can apply \eqref{est_2} along with \eqref{suplim}, obtaining that $\det \nabla w_\e \to 1$ strongly in $L^\frac{p}{3}(\Omega)$. To be precise, we infer $\|\det \nabla w_\e - 1\|_{L^{p/3}(\Omega)} \le C \e^\alpha$ every $\e > 0$. In particular, Vitali's Convergence Theorem ensures that, since $|\Omega| < +\infty$, the sequence $(\det\nabla w_\e)_\e$ is equi-integrable in $\Omega$.

    
    \smallskip
    {\bf{Step 3: Convergence of the inverses.}} First, we prove that the family $(1/\det \nabla w_\e)_\e$ is equi-integrable. Note that, since $w_\e \in \Y$, see \eqref{eq:class-Y}, for every $\e > 0$ it makes sense to consider the inverse of the determinant. Let us define the function $\tilde h : (0, +\infty) \to (0, +\infty)$ by setting
    \begin{equation*}
        \tilde h(t) := h \left(\frac{1}{t}\right).
    \end{equation*}
    Thanks to \eqref{grwt-h}, we have that
    \begin{equation}
    \label{aux-comp}
        \lim_{t \to +\infty} \frac{\tilde h(t)}{t} = \lim_{t \to +\infty} \frac{1}{t} \, h \left(\frac{1}{t}\right) = \lim_{k \to 0^+} k h(k) = +\infty.
    \end{equation}
    Moreover, thanks to \ref{H2}, \eqref{suplim} and the definition of $\tilde h$ we have that, for every $\e > 0$,
    \begin{equation*}
        \int_\Omega \tilde h \left(\frac{1}{\det \nabla w_\e}\right) \dx = \int_\Omega h(\det \nabla w_\e) \dx \le \e^{2 \alpha} F_\e(u_\e, m_\e) < + \infty.
    \end{equation*}
    Applying De la Vallée--Poussin criterion (see, e.g., \cite[Theorem 1.3]{ekeland-temam}), it follows that $(1/\det \nabla w_\e)_\e$ is equi-integrable, as desired.
    \\
    Now, for any sufficiently small fixed $\delta > 0$, we denote for any $\e > 0$
    \begin{equation}
    \label{B_eps_d}
        B^\delta_\e := \{ x \in \Omega \ : \ |\det \nabla w_\e-1| > \delta \}.
    \end{equation}
    It is not difficult to see that, by Step 2 of the proof,
    \begin{equation*}
        |B^\delta_\e| \le C \delta^{-1} \|\det \nabla w_\e - 1\|_{L^{p/3}(\Omega)}^\frac{p}{3} \le C \delta^{-1} \e^{\alpha \frac{p}{3}}.
    \end{equation*}
    By the equi-integrability of $(|\det\nabla w_\e|^{-1})_\e$, we have 
    \begin{equation*}
        \int_{B^\delta_\e} \left| |\det\nabla w_\e|^{-1} - 1 \right| \dx \to 0.
    \end{equation*}
    Moreover,
    \begin{equation*}
        \int_{\Omega \setminus B^\delta_\e} \left| |\det\nabla w_\e|^{-1} - 1 \right| \dx \le C \int_{\Omega \setminus B^\delta_\e} \left| |\det \nabla w_\e| - 1 \right| \dx \le C \e^\alpha,
    \end{equation*}
    which yields the claim.
    
    \smallskip
    {\bf{Step 4: Weak convergence of the compositions.}} First, we show that
    \begin{equation}
    \label{5.14}
        \|\nabla (m_\e \circ w_\e)\|_{L^q(\Omega)} \le C \quad \textnormal{ with } q := \frac{2ps}{p+s(p+2)} \in \left(1, \frac{2p}{p+2} \right).
    \end{equation}
    Let us, at first, consider $q \in (1,2)$. For $l := 2/q$, and $l':=\tfrac{2}{2-q}$ its conjugate exponent, we have by H\"older's inequality
    \begin{equation*}
    \begin{aligned}
        \int_\Omega |\nabla (m_\e \circ w_\e)|^q \dx
        &\le \left( \int_\Omega \left|\nabla m_\e \circ w_\e (\det\nabla w_\e)^\frac{1}{2}\right|^{ql} \dx \right)^\frac{1}{l}
        \left( \int_\Omega \left|\nabla w_\e (\det\nabla w_\e)^{-\frac{1}{2}}\right|^{ql'} \dx \right)^\frac{1}{l'} \\
        &= \left( \int_\Omega \left|\nabla m_\e \circ w_\e (\det\nabla w_\e)^\frac{1}{2}\right|^2 \dx \right)^\frac{q}{2}
        \left( \int_\Omega \left|\nabla w_\e (\det\nabla w_\e)^{-\frac{1}{2}}\right|^\frac{2q}{2-q} \dx \right)^{\frac{2-q}{2}}. 
    \end{aligned}
    \end{equation*} 
    We know already that the first term is bounded by \eqref{suplim}. Further, let us choose $\mu := \frac{p(2-q)}{2q}$. Note that
    \begin{equation*}
        \mu > 1 \quad \iff \quad q < \frac{2p}{p+2} \in (1,2),
    \end{equation*}
    meaning that it allows us to restrict the choice of $q$, namely
    \begin{equation*}
        q \in \left(1, \frac{2p}{p+2} \right).
    \end{equation*}
    Again by means of H\"older's inequality, using $\mu$ and its conjugate exponent $\mu' := \tfrac{p(2-q)}{p(2-q)-2q}$, we get
    \begin{equation*}
    \begin{aligned}
        \int_\Omega \left|\nabla w_\e (\det\nabla w_\e)^{-\frac{1}{2}}\right|^\frac{2q}{2-q} \dx 
        &\le \left( \int_\Omega |\nabla w_\e|^\frac{2q\mu}{2-q} \dx \right)^\frac{1}{\mu}
        \left( \int_\Omega |\det\nabla w_\e|^{-\frac{q\mu'}{2-q}} \dx \right)^\frac{1}{\mu'} \\
        &= \left( \int_\Omega |\nabla w_\e|^p \dx \right)^\frac{2q}{p(2-q)} \left( \int_\Omega  |\det\nabla w_\e|^{-\frac{pq}{p(2-q) - 2q}} \dx \right)^\frac{p(2-q) - 2q}{p(2-q)}.
    \end{aligned}
    \end{equation*}
    By Step 1 of the proof, we know that the first term is bounded. Moreover, from \eqref{5.15a} we know that $|\det\nabla w_\e|^{-1} \in L^s(\Omega)$, yielding that \eqref{5.14} holds if
    \begin{equation*}
        s = \frac{pq}{p(2-q)-2q} \quad \iff \quad q := \frac{2ps}{p+s(p+2)}.
    \end{equation*}
    Note that the condition $s > \tfrac{p}{p-2}$ implies that $q > 1$, while it is immediate to verify that
    \begin{equation*}
        q := \frac{2ps}{p+s(p+2)} < \frac{2p}{p+2}.
    \end{equation*}
    This, along with the constraint $|m_\e \circ w_\e| = 1$, entails that we have that $\| m_\e \circ w_\e\|_{W^{1,q}(\Omega;\R^3)} \le C$, with $q$ defined as above, implying that $m_\e \circ w_\e \rightharpoonup m$ weakly in $W^{1,q}(\Omega; \R^3)$.
    
    \smallskip
    {\bf{Step 5: Strong convergence of the compositions.}} In view of Step 4 of the proof, an application of the Rellich-Kondrachov Theorem ensures that 
    \begin{equation*}
        m_\e \circ w_\e \to m \qquad \textnormal{ strongly in } L^\sigma(\Omega; \R^3) \ \textnormal{ for any } \sigma \in [1, q^*).
    \end{equation*}
    Due to the constraint $|m_\e \circ w_\e| = 1$ a.e. in $\Omega$, $m_\e \circ w_\e$ turns out to be bounded in any $L^r(\Omega; \R^3)$, with $r \in [1, +\infty)$. Therefore, Lemma \ref{l5.5} implies that $m_\e \circ w_\e$ strongly converges to $m$ in $L^r(\Omega; \R^3)$ for any $r \in [1, +\infty)$.
    \\
    Eventually, since in particular $m_\e \circ w_\e \to m$ strongly in $L^1(\Omega; \R^3)$, we get that, up to a subsequence, $m_\e \circ w_\e \to m$ a.e. in $\Omega$. Therefore, since $(u_\e, m_\e) \in \A_\e$, $|m_\e\circ w_\e| = 1$, it allows us to conclude that $m \in W^{1,q}(\Omega; \s^2)$.
    
    \smallskip
    {\bf{Step 6: Strong convergence of the extensions.}}
    In view of Lemma \ref{l5.5}, in order to prove that $\chi_{w_\e(\Omega)} m_\e \to \chi_\Omega m$ strongly in $L^r(\R^3; \R^3)$ for any $r \in [1,+\infty)$ it is enough to show the boundedness of $\chi_{w_\e(\Omega)} m_\e$ in $L^r(\R^3; \R^3)$ for any $r \in [1, +\infty)$ along with the strong convergence in $L^1$, namely
    \begin{equation}
    \label{eq:convergenceinL1}
        \|\chi_{w_\e(\Omega)} m_\e - \chi_\Omega m\|_{L^1(\R^3; \R^3)} \to 0 \qquad \textnormal{ as } \e \to 0.
    \end{equation}
    The boundedness of $\chi_{w_\e(\Omega)} m_\e$ in $L^r(\R^3; \R^3)$ for any $r \in [1, +\infty)$ follows from the saturation constraint $|m_\e \circ w_\e| = 1$ and the boundedness of $\det \nabla w_\e$ in $L^1(\Omega)$. Indeed, for any fixed $r \in [1, +\infty)$, 
    \begin{equation}
    \label{eq:boundextension}
    \begin{aligned}
        \int_{\R^3} |\chi_{w_\e(\Omega)} m_\e|^r \dif z &= \int_{w_\e(\Omega)} |m_\e|^r \dif z = \int_\Omega |m_\e(w_\e(x))|^r \det \nabla w_\e \dx \\
        &= \int_\Omega \det\nabla w_\e \dx < +\infty.
    \end{aligned} 
    \end{equation}
    We now turn to prove \eqref{eq:convergenceinL1}. To this end, for a fixed $\delta > 0$, we set
    \begin{equation*}
        \Omega_\delta := \{x \in \Omega \ : \ \dist(x, \partial\Omega) > \delta \}.
    \end{equation*}
    Without loss of generality, we may assume that $\Omega_\delta$ is a Lipschitz set. Since $w_\e \to id$ in $W^{1,p}(\Omega; \R^3)$ and $p > 3$, we have that $w_\e \to id$ in $C^0(\overline{\Omega}; \R^3)$. Hence, for $\e > 0$ small enough, we may assume that $\Omega_\delta \subset w_\e(\Omega)$ (thanks to the degree theory, see \cite{FG} and \cite[Proof of Proposition 3.1]{AKM24}). Now, 
    \begin{equation}
    \label{eq:convL1split}
        \|\chi_{w_\e(\Omega)} m_\e - \chi_\Omega m\|_{L^1(\R^3; \R^3)} = \int_{\Omega_\delta} |\chi_{w_\e(\Omega)} m_\e - \chi_\Omega m| \dx + \int_{\R^3 \setminus \Omega_\delta} |\chi_{w_\e(\Omega)} m_\e - \chi_\Omega m| \dx.
    \end{equation}
    The second integral on the right-hand side can be rendered negligible as $\delta \to 0$. Indeed, since $\chi_{w_\e(\Omega)} m_\e$ is bounded in $L^r(\R^3; \R^3)$ for any $r \in [1, +\infty)$, we infer that
    \begin{equation*}
        \int_{\R^3 \setminus \Omega_\delta} |\chi_{w_\e(\Omega)} m_\e - \chi_\Omega m| \dx \le C \mathcal{L}^3( w_\e(\Omega) \, \Delta \, \Omega_\delta) \to C \mathcal{L}^3( w_\e(\Omega)\, \Delta \, \Omega) \quad \textnormal{ as } \delta \to 0.
    \end{equation*}
    As for the first integral on the right-hand side of \eqref{eq:convL1split},  note that 
    \begin{align}
        \int_{\Omega_\delta}|\chi_{w_\e(\Omega)}m_\e- \chi_{\Omega}m| \dx 
        &\le \int_{\Omega_\delta} |\chi_{w_\e(\Omega)}m_\e - m_\e\circ w_\e| \dx + \int_{\Omega_\delta}|m_\e\circ w_\e - \chi_{\Omega}m| \dx \notag \\
        &\le \int_{\Omega_\delta} |\chi_{w_\e(\Omega)}m_\e - m_\e\circ w_\e| \dx + \int_{\Omega}|m_\e\circ w_\e - m| \dx.
    \label{eq:convOmegadelta}
    \end{align}
    Here, we have used the fact that $\Omega_\delta\subseteq \Omega\cap w_\e(\Omega)$. By Step 5, we can easily deduce that the second integral in \eqref{eq:convOmegadelta} converges to $0$ as $\e\to 0$. Furthermore, since $\chi_{w_\e(\Omega)} m_\e = m_\e$ a.e. in $\Omega_\delta$, the first integral in \eqref{eq:convOmegadelta} turns into the following
    \begin{equation}
    \label{eq:estfirstint1}
        \int_{\Omega_\delta} |\chi_{w_\e(\Omega)}m_\e - m_\e\circ w_\e| \dx = \int_{\Omega_\delta} |m_\e - m_\e\circ w_\e| \dx.
    \end{equation}
    In view of Steps 2 and 5, the Dunford-Pettis Theorem enables us to have the equi-integrability of $\det \nabla m_\e$ as well as $m_\e\circ w_\e$. Furthermore, from \eqref{eq:boundextension}, it follows that $\chi_{\Omega_\delta}m_\e$ is equi-integrable in $\Omega_\delta$ too, indeed
    \begin{equation*}
    \begin{aligned}
        \int_{\R^3} |\chi_{\Omega_\delta} m_\e|^r \dif z &= \int_{\Omega_\delta} |m_\e|^r \dif z \le \int_{w_\e(\Omega)} |m_\e|^r \dif z \\
        &= \int_\Omega |m_\e(w_\e(x))|^r \, \det \nabla w_\e  \dx = \int_\Omega \det\nabla w_\e\dx < +\infty.
    \end{aligned} 
    \end{equation*}
    Therefore, for $\eta>0$ there exists $\rho>0$ such that for any $A\subseteq\Omega_\delta$ measurable 
    \begin{equation*}
        \mathcal{L}^3(A)<\rho \ \Rightarrow \ \limsup_{\e\to 0} \int_A |m_\e\circ w_\e - m_\e| \dx < \eta. 
    \end{equation*}
    Fix $\overline{\delta}\in (0, \delta)$ such that $\mathcal{L}^3(\Omega_{\overline{\delta}} \setminus \Omega_\delta)<\rho$ and set $A_{\delta, \e} := w^{-1}_{\e}(\Omega_\delta)$. For $\e>0$ small enough, we have that $A_{\delta, \e}\subseteq \Omega_{\overline{\delta}}$.  Therefore, 
    \begin{equation}
    \label{eq:estfirstint2}
        \int_{\Omega_\delta} |m_\e\circ w_\e - m_\e| \dx = \int_{\Omega_\delta\cap A_{\delta, \e}} |m_\e\circ w_\e - m_\e| \dx + \int_{\Omega_\delta\setminus A_{\delta, \e}} |m_\e\circ w_\e - m_\e| \dx.
    \end{equation}
    The integral over $\Omega_\delta\setminus A_{\delta, \e}$ can be estimated as follows. First, note that a change of variables implies that 
    \begin{equation*}
        \mathcal{L}^3(\Omega_\delta) = \mathcal L^3 (w_\e(A_{\delta, \e})) = \int_{w_\e(A_\delta, \e)} \dx = \mathcal L^3(A_{\delta, \e}) + \int_{\A_{\delta, \e}} ({\rm det}\nabla w_\e -1) \dx.
    \end{equation*}
    Therefore,
    \begin{equation*}
    \begin{aligned}
        \limsup_{\e\to 0} \, \mathcal L^3(\Omega_\delta\setminus A_{\delta, \e})
        &\le \limsup_{\e\to 0} \, \left[\mathcal{L}^3 (\Omega_{\overline{\delta}})- \mathcal{L}^3(A_{\delta, \e}) \right] \\
        &\le \limsup_{\e\to 0} \left[\mathcal L^3(\Omega_{\overline{\delta}}\setminus\Omega_\delta) + \int_{A_{\delta, \e}} (\det\nabla w_\e -1) \dx \right] \\
        &< \rho.
    \end{aligned}
    \end{equation*}
    Here, we have used Step 2 as well as $A_{\delta, \e} \subseteq \Omega_{\overline{\delta}}$. Thanks to the equi-integrability of $m_\e\circ w_\e$ and $m_\e$ on $\Omega_\delta$, we conclude that 
    \begin{equation}
    \label{eq:estfirstint3}
        \limsup_{\e\to 0} \int_{\Omega_{\delta}\setminus A_{\delta, \e}} |m_\e\circ w_\e - m_\e| \dx < \eta.
    \end{equation}
    We now evaluate the integral over $\Omega_{\delta}\cap A_{\delta, \e}$ in \eqref{eq:estfirstint2}. To this end, we fix $\lambda\in (-\tfrac{1}{\rho}, 0)$ and extend $\chi_{\Omega_\delta}m_\e$ from $\Omega_\delta$ to a map $M_\e^\delta\in W^{1,2}(\R^3; \R^3)$ with $\|M_\e^\delta\|_{W^{1,2}(\R^3; \R^3)} \le C(\Omega_\delta) \|\chi_{\Omega_\delta}m_\e\|_{W^{1,2}(\Omega_\delta; \R^3)}$. Then, we apply the Lusin approximation of Sobolev functions: for every $\e>0$ there exists a set $G_{\delta, \e}$ such that $M_\e^\delta$ is $\e^\lambda$-Lipschitz on $G_{\delta, \e}$ and 
    \begin{equation}
    \label{eq:Lusin1}
        \mathcal{L}^3(\R^d \setminus G_{\delta, \e}) \le \e^{-2\lambda} \int_{\R^d} |\nabla M_\e^\lambda|^2 \dx \le C(\Omega_\delta) \e^{-2\lambda} \|\chi_{\Omega_\delta}m_\e\|_{W^{1,2}(\Omega_\delta; \R^3)}^2.
    \end{equation}
    Set $X_{\delta, \e} := w_\e^{-1}(G_{\delta, \e})$. Since $m_\e = M_\e^\delta$ on $\Omega_\delta \cap G_{\delta, \e}$ and keeping in mind convergence rate \eqref{eq:convrate}, it follows that
    \begin{equation*}
    \begin{aligned}
        \int_{\Omega_\delta\cap A_{\delta, \e}\cap X_{\delta, \e}\cap G_{\delta, \e}} |m_\e\circ w_\e -  m_\e| \dx 
        & = \int_{\Omega_\delta\cap A_{\delta, \e}\cap X_{\delta, \e}\cap G_{\delta, \e}} |M_\e^\delta \circ w_\e - M_\e^\delta| \dx \\
        &\le \e^\lambda \mathcal{L}^3(\Omega) \| w_\e - id\|_{C^0} \\
        &\le \e^{2 \tfrac{\alpha}{p} + \lambda} \mathcal{L}^3(\Omega).
    \end{aligned}
    \end{equation*}
    Deploying \eqref{eq:Lusin1} along with \eqref{eq:convrate} and a change of variables, we deduce that
    \begin{equation*}
        \lim_{\e\to 0} \mathcal L^3(\Omega_{\delta} \setminus G_{\delta, \e})=0 \qquad \textnormal{and} \qquad \lim_{\e\to 0} \mathcal L^3 (\Omega_\delta\setminus X_{\delta, \e})=0.
    \end{equation*}
    Thus, 
    \begin{equation}
    \label{eq:estfirstint4}
        \limsup_{\e\to 0} \int_{(\Omega_\delta \cap A_{\delta, \e}) \setminus (X_{\delta, \e\cap G_{\delta, \e}})} |m_\e^\delta\circ w_\e -  m_\e| \dx \le \eta.
    \end{equation}
    Bearing in mind \eqref{eq:convOmegadelta} and gathering \eqref{eq:estfirstint1}, \eqref{eq:estfirstint2}, \eqref{eq:estfirstint3} and \eqref{eq:estfirstint4}, we conclude that 
    \begin{equation*}
        \limsup_{\e\to 0} \int_{\Omega_\delta} |m_\e\circ w_\e - m| \dx < 3\eta.
    \end{equation*}
    Since $\eta>0$ is arbitrary, we get \eqref{eq:convergenceinL1}.
    
    \smallskip
    {\bf{Step 7: Convergence of the extensions of the gradients.}}  
    First, note that $\nabla(\chi_{\Omega_\delta}m_\e)$ is bounded in $L^2(\Omega_\delta; \R^{3\times 3})$. Indeed, 
    \begin{equation}
    \label{eq:boundgradextension}
        \int_{\Omega_\delta} |\nabla(\chi_{\Omega_\delta} m_\e)|^2 \dx \le \int_{\Omega_\delta} |\nabla m_\e|^2 \dx \le \int_{w_\e(\Omega)} |\nabla m_\e|^2 \dx \le \frac{1}{C_1} \, F_\e(u_\e, m_\e) < +\infty.
    \end{equation}
    Moreover, by \eqref{eq:boundextension}, $\chi_{\Omega_\delta}m_\e$ turns out to be bounded in $L^2(\Omega_\delta; \R^3)$. Thus, $\chi_{\Omega_\delta}m_\e$ is bounded in $W^{1,2}(\Omega_\delta; \R^3)$ uniformly for all $\delta>0$.  By a diagonal argument, there exists $\eta\in W^{1,2}_{\loc}(\Omega; \R^3)$ such that 
    \begin{equation*}
        \chi_{\Omega_\delta}m_\e \rightharpoonup  \eta \qquad \textnormal{weakly in } W^{1,2}_{\loc}(\Omega; \R^3).
    \end{equation*}
    From Step 6, it follows that $\eta = \chi_{\Omega}m$. In particular, we have that $m_\e\rightharpoonup m$ weakly in $W^{1,2}(\Omega_\delta; \R^3)$.
    \\
    We now show the weak convergence of $\chi_{w_\e(\Omega)}\nabla(m_\e)$. To this end, for a test function $\varphi \in L^2(\R^3; \R^{3 \times 3})$, we deduce that   
    \begin{equation*}
    \begin{aligned}
        &\int_{\R^3} (\chi_{w_\e(\Omega)} \nabla m_\e - \chi_\Omega \nabla m) : \varphi \dx \\
        &\quad = \int_{\Omega_\delta} (\chi_{w_\e(\Omega)} \nabla m_\e - \chi_\Omega \nabla m) : \varphi \dx + \int_{\R^3 \setminus \Omega_\delta} (\chi_{w_\e(\Omega)} \nabla m_\e - \chi_\Omega \nabla m) : \varphi \dx.
    \end{aligned}
    \end{equation*}
    Since $\Omega_\delta \subset w_\e(\Omega) \cap \Omega$ and due the fact that $m_\e=\chi_{\Omega_\delta} m_\e \rightharpoonup m_0$ weakly in $W^{1,2}(\Omega_\delta; \R^3)$, the first integral goes to $0$. As for the second integral, we can apply the same arguments as in \cite{AKM24}. Namely, since
    \begin{equation*}
        \left| \int_{\R^3 \setminus \Omega_\delta} (\chi_{w_\e(\Omega)} \nabla m_\e - \chi_\Omega \nabla m_0) : \varphi \dx \right| \le C \| \varphi \|_{L^2(C_\delta)},
    \end{equation*}
    where $C_\delta := \{ x \in \R^3 \ : \ \dist(x, \partial \Omega) < \delta \}$. Since $|C_\delta| \to 0$ as $\delta \to 0$, we have the weak convergence of $\chi_{w_\e(\Omega)}\nabla m_\e$ to $\chi_\Omega\nabla m_0$ in $L^2(\R^3; \R^{3 \times 3})$.   
\end{proof}



\section{Two-scale convergence of the magnetizations}
\label{sec_2sc}

Before proving the main result, we need to investigate the two-scale behaviour of the gradient of the magnetization. The main result of this section reads as follows.

\begin{prop}
\label{prop:two_scale}
    Let $\alpha > 0$. Let $(u_\e, m_\e) \in \A_\e$ be such that $w_\e := id + \e^\alpha u_\e \in \Y$, and let
    \begin{equation*}
        \sup_{\e > 0} F_\e(u_\e, m_\e) < + \infty,
    \end{equation*}
    where the energy $F_\e$ is defined as in \eqref{resc-en}. Then we have up to a subsequence that
    \begin{equation}
    \label{5.10}
		\nabla m_\e \circ w_\e |\det\nabla w_\e|^\frac{1}{2} \wts \nabla m(x) + \nabla_y \phi(x,y) \quad \textnormal{ weakly two-scale in } L^2(\Omega \times Y; \R^3),
    \end{equation}
    where $\phi \in L^2(\Omega; H^1_\per(Y; \R^3))$.
    \qed
\end{prop}

In order to prove this Proposition, we need some preliminary analysis on the quantities we are working with. Now, observe that $\nabla m_\e \in L^2(w_\e(\Omega))$ if and only if $\nabla m_\e \circ w_\e |\det\nabla w_\e|^\frac{1}{2} \in L^2(\Omega)$. Note that
\begin{equation*}
    \nabla (m_\e \circ w_\e) = \nabla m_\e \circ w_\e \, \nabla w_\e,
\end{equation*} 
thus
\begin{equation}
\label{5.14a}
    \nabla m_\e \circ w_\e |\det\nabla w_\e|^\frac{1}{2} = \nabla (m_\e \circ w_\e) (\nabla w_\e)^{-1} |\det\nabla w_\e|^\frac{1}{2}.
\end{equation}
Now for brevity let us denote   
\begin{equation}
\label{def_g_e}
    g_\e := (\nabla w_\e)^{-1} |\det\nabla w_\e|^\frac{1}{2} = {\rm adj}(\nabla w_\e) |\det\nabla w_\e|^{-\frac{1}{2}}.
\end{equation}

\begin{lemma}
\label{lmm:conv_geps}
    Under the assumptions of Proposition \ref{prop:two_scale}, we have
    \begin{equation*}
        g_\e \to \Id \ \textnormal{ in } L^{r^*}(\Omega; \R^{3 \times 3}), \quad \textnormal{ where  } r^* < r := \frac{ps}{p+2s} \in \left(1, \frac{p}{2}\right), 
    \end{equation*}
    and where $s$ is the exponent in \eqref{5.15a}.
    \qed
\end{lemma}

In order to prove this result, we want to exploit Lemma \ref{l5.5} and we need the following auxiliary result.

\begin{lemma}
\label{l5.6}
    Under the assumptions of Proposition \ref{prop:two_scale}, we have
    \begin{equation*}
        g_\e \to \Id \ \textnormal{ in } \ L^1(\Omega; \R^{3 \times 3}).
    \end{equation*} 
\end{lemma}
\begin{proof}
    Note that
    \begin{equation}
    \label{5.16}
        |g_\e - \Id| \le |{\rm adj}(\nabla w_\e) - \Id| |\det\nabla w_\e|^{-\frac{1}{2}} + \left||\det\nabla w_\e|^{-\frac{1}{2}} - 1\right|,
    \end{equation}
    and that
    \begin{equation*}
        \left| |\det\nabla w_\e|^{-\frac{1}{2}} - 1 \right| \le \left| |\det\nabla w_\e|^{-1} - 1 \right|.
    \end{equation*}
    We know from Proposition \ref{compact} that, under the hypotheses of the Lemma and under the assumption that $\underset{\e>0}{\sup} \, F_\e(u_\e, m_\e) < +\infty$, $w_\e \to id$ in $W^{1,p}(\Omega; \R^3)$, which immediately implies that
    \begin{equation}
    \label{5.16a}
        {\rm adj}(\nabla w_\e) \to \Id \ \textnormal{ in } L^\frac{p}{2}(\Omega; \R^{3 \times 3}).
    \end{equation}
    Then the claim follows by Proposition \ref{compact} by integrating \eqref{5.16} and applying H\"older's inequality.
\end{proof}

\begin{remark}
    Note that, in order to get $g_\e\in L^1(\Omega; \R^{3 \times 3})$, it would be sufficient to assume $|\det\nabla w_\e|^{-1} \in L^{\frac{p}{2(p-2)}}(\Omega)$ instead of \eqref{5.15a}.
\end{remark}

Now we are ready to give the proof of Lemma \ref{lmm:conv_geps}.

\begin{proof}[Proof of Lemma \ref{lmm:conv_geps}]
    We now want to prove that $g_\e\in L^r(\Omega; \R^{3 \times 3})$ for some $r > 1$. Indeed, applying the generalized H\"older's inequality and using \eqref{5.15a} we get
    \begin{equation*}
    \begin{aligned}
        \|g_\e\|_{L^r(\Omega; \R^{3 \times 3})}
        &\le \||{\rm adj}(\nabla w_\e)|\|_{L^\frac{p}{2}(\Omega)} \||\det\nabla w_\e|^{-1}\|_{L^{\frac{pr}{p-2r}}(\Omega)} \\
        &= \||{\rm adj}(\nabla w_\e)|\|_{L^\frac{p}{2}(\Omega)} \||\det\nabla w_\e|^{-1}\|_{L^s(\Omega)} < +\infty,
    \end{aligned}
    \end{equation*}
    where we chose $r := \tfrac{ps}{p+2s} \in \left(1, \tfrac{p}{2}\right)$. Now, since we are under the hypotheses of Proposition \ref{compact}, \eqref{5.16a} holds. This entails that we can apply Lemma \ref{l5.6}, which together with Lemma \ref{l5.5} implies the thesis.
\end{proof}
		
We are now ready to give the proof of Proposition \ref{prop:two_scale}.

\begin{proof}[Proof of Proposition \ref{prop:two_scale}]
    Let us first define $\hat{g}_\e$ by components as
    \begin{equation*}
        (\hat{g}_\e)_{ij}(x) := \min\{|(g_\e)_{ij}(x)|,\, 2\} \, {\rm sgn} (g_\e)_{ij}(x),
    \end{equation*}
    where $g_\e$ is defined as in \eqref{def_g_e}. Clearly, $|\hat{g}_\e - \Id| \le |g_\e - \Id|$. By Lemmas \ref{l5.5} and \ref{l5.6} we get that 
    \begin{equation}
    \label{5.15}
        \hat{g}_\e \to \Id \ \textnormal{ in } \ L^\sigma(\Omega; \R^{3 \times 3}) \ \textnormal{ for any } \sigma \in \left[ 1, +\infty \right).
    \end{equation}
    Further, again from Lemma \ref{l5.6}, it follows that $|\supp (g_\e - \hat{g}_\e)| \to 0$ as $\e \to 0$. Then, thanks to \eqref{5.14a} and \eqref{def_g_e}, we have that
    \begin{equation*}
    \begin{aligned}
        |\nabla (m_\e \circ w_\e) (g_\e - \hat{g}_\e)|
        &\le |\nabla (m_\e \circ w_\e) g_\e| + |\nabla (m_\e \circ w_\e) \hat{g}_\e| \\
        &\le |\nabla (m_\e \circ w_\e) g_\e| + C|\nabla (m_\e \circ w_\e)| \\
        &= \left|\nabla m_\e \circ w_\e |\det\nabla w_\e|^\frac{1}{2}  \right| + C|\nabla (m_\e \circ w_\e)|
    \end{aligned}
    \end{equation*}
    which, together with the fact that the two terms on the right hand side of the latter are bounded in $L^2(\Omega)$ and $L^q(\Omega)$ respectively, with $q$ as in Proposition \ref{compact}, entails that, for any test function $\psi\in C(\overline{\Omega\times Y})$, the sequence $\left\{\nabla (m_\e \circ w_\e) (g_\e - \hat{g}_\e) \, \psi\left(x,\frac{x}{\e}\right)\right\}$ is equi-integrable. This remark, combined with the vanishing support of $\{g_{\e}-\hat{g}_e\}$ yields
    \begin{equation}
    \label{5.20}
        \int_\Omega \nabla (m_\e \circ w_\e) (g_\e - \hat{g}_\e) \, \psi\left(x,\frac{x}{\e}\right) \dx \to 0 \quad \textnormal{ as } \ \e \to 0
    \end{equation} 
    for every $\psi\in C(\overline{\Omega\times Y})$. By \eqref{5.15}, for $\psi$ as before, one has
    \begin{equation*}
        \hat{g}_\e \psi\left(x,\frac{x}{\e}\right) \sts \psi(x,y) \quad \textnormal{ strongly two-scale in } L^\sigma(\Omega \times Y; \R^{3 \times 3}) \textnormal{ for any } \sigma \in \left[1, +\infty\right).
    \end{equation*} 
    Then, by properties of two-scale convergence, 
    \begin{equation*}
        \int_\Omega \nabla (m_\e \circ w_\e)   \hat{g}_\e \, \psi\left(x,\frac{x}{\e}\right) \dx \to \int_{\Omega}\int_{Y} (\nabla m(x) + \nabla_y \phi(x,y)) \psi(x,y) \dx. 
    \end{equation*} 
    This, together with \eqref{5.14a} and \eqref{5.20}, yields \eqref{5.10}.
\end{proof}



\section{Proof of Theorem \ref{main_thm}}
\label{gamma-conv}
This section is devoted to the proof of Theorem \ref{main_thm}. In particular, in Subsection \ref{lim_magnetostatic} we deal with the magnetostatic self-energy term, while in Subsection \ref{low-bou} we prove the lower bound for our family of energy functionals. In Subsection \ref{conv_rec_seq} we then construct the recovery sequences for the energies and we prove their convergences, showing that they satisfy the ones in Proposition \ref{compact}, while lastly in Subsection \ref{rec_seq} we show that the limiting functional is recovered through these sequences and we finally prove Theorem \ref{main_thm}.


\subsection{Limit of the magnetostatic self-energy term}
\label{lim_magnetostatic}

Regarding the magnetostatic self-energy term of the energy functional $F_\e$ given by \eqref{resc-en}, we show something more than the expected lower bound, in particular that
\begin{equation}
\label{eq_magnetostatic}
    \lim_{\e \to 0} \frac{\mu_0}{2} \int_{\R^3} |\nabla \psi_{m_\e}|^2 \dx = \frac{\mu_0}{ 2} \int_{\R^3} |\nabla \psi_m| \dx.
\end{equation}
Let $(u_\e, m_\e) \in \A_\e$ be such that $w_\e = id + \ealpha u_\e \in \Y$ and $\sup_{\e>0} F_\e(u_\e, m_\e) < +\infty$. Given $\chi_{\omega_\e(\Omega)} m_\e \in L^2(\R^3;\R^3)$, thanks to Lemma \ref{lmm:sol_magneto_self} we know that there exists $\psi_{m_\e} \in L^{1,2}(\R^3)$ such that it is the unique solution, up to additive constants, to 
\begin{equation*}
    \int_{\R^3} \chi_{\omega_\e(\Omega)} m_\e \cdot \nabla \varphi \dx = \int_{\R^3} \nabla\psi_{m_\e} \cdot \nabla \varphi \dx \quad \forall \, \varphi \in L^{1,2}(\R^3),
\end{equation*}
and such that
\begin{equation}
\label{aprioriestimate}
    \|\psi_{m_\e}\|_{L^{1,2}(\R^3)} = \|\nabla \psi_{m_\e}\|_{L^2(\R^3 \times \R^3)} \le \|\chi_{\omega_\e(\Omega)} m_\e\|_{L^2(\R^3)}.
\end{equation}
By Proposition \ref{compact}, we have that $\chi_{\omega_\e(\Omega)}m_\e$ is bounded in $L^2(\R^3;\R^3)$ and this, along with \eqref{aprioriestimate}, yields 
\begin{equation*}
    \nabla \psi_{m_\e} \rightharpoonup \nabla \psi_m \quad \textnormal{ weakly in } L^2(\R^3; \R^3),
\end{equation*}
where $\psi_m \in L^{1,2}(\R^3)$ is the unique solution, up to additive constants, to 
 \begin{equation*}
    \int_{\R^3} \chi_\Omega m \cdot \nabla\varphi \dx = \int_{\R^3} \nabla\psi_m \cdot \nabla\varphi \dx \quad \forall \, \varphi\in L^{1,2}(\R^3).
\end{equation*}
Now, thanks to the linearity of the Maxwell equation \eqref{maxwell}, $\psi_{m_\e}-\psi_{m}$ is the unique solution, again up to additive constants, to
\begin{equation*}
    \int_{\R^3} (\chi_{\omega_\e(\Omega)} m_\e - \chi_{\Omega} m) \cdot \nabla\varphi \dx = \int_{\R^3} (\nabla\psi_{m_\e} - \nabla\psi_m) \cdot \nabla\varphi \dx \quad \forall \, \varphi\in L^{1,2}(\R^3).
\end{equation*}
Note that, again by Lemma \ref{lmm:sol_magneto_self}, it also holds that
\begin{equation*}
    \|\nabla \psi_{m_\e} - \nabla \psi_m  \|_{L^2(\R^3;\R^3)} \le \|\chi_{\omega_\e(\Omega)} m_\e - \chi_\Omega m \|_{L^2(\R^3;\R^3)},
\end{equation*}
which in turn implies, thanks to Proposition \ref{compact}, that $\nabla \psi_{m_\e}$ strongly converges in $L^2(\R^3;\R^3)$ to $\nabla \psi_m$. Therefore \eqref{eq_magnetostatic} holds.


\subsection{Lower bound}
\label{low-bou}

In order to deal with the lower bound for our family of energy functionals, let us firstly remark, by the superadditivity of the $\liminf$, that
\begin{equation*}
\begin{aligned}
    \liminf_{\e \to 0} F_\e (u_\e, m_\e)
    \ge & \, \liminf_{\e\to 0} \frac{1}{\e^{2\alpha}} \int_\Omega W \left(\frac{x}{\e}, \Id + \e^\alpha \nabla u_\e(x), m_\e(x + \e^\alpha u(x)) \right) \dx \\
    + & \, \liminf_{\e \to 0} \int_\Omega a \left(\frac{x}{\e}\right) |\nabla m_\e(x + \e^\alpha u_\e(x))|^2 |\det\nabla w_\e| \dx \\
    + & \, \liminf_{\e \to 0} \, \frac{\mu_0}{2} \int_{\R^3} |\nabla \psi_{m_\e}|^2 \dx.
\end{aligned}
\end{equation*}
We already dealt with the magnetostatic self-energy term of the energy, see Subsection \ref{lim_magnetostatic}, therefore it is enough to compute separately the lower bounds of each of the two remaining terms. Let $(u_\e, m_\e) \in \A_\e$ be such that $w_\e = id + \e^\alpha u_\e \in \Y$. From this point on, we can assume $\sup_{\e > 0} F_\e(u_\e, m_\e) < +\infty$, otherwise the result is obvious, where the energy $F_\e$ is defined as in \eqref{resc-en}.

\smallskip
{\bf{Step 1: Lower bound for the elastic term.}}
Regarding the elastic term of the energy functional, we show that
\begin{equation*}
    \liminf_{\e \to 0} \frac{1}{\e^{2\alpha}} \int_\Omega W \left(\frac{x}{\e}, \Id + \e^\alpha \nabla u_\e(x), m_\e(x + \e^\alpha u_\e(x)) \right) \dx \ge \frac{1}{2} \int_\Omega Q_\hom(\nabla u(x), m(x)) \dx,
\end{equation*}
with $Q_\hom$ being defined as in \eqref{Qhom}. To this end, we adapt the strategy used in \cite[Theorem 5.2.1]{N10}, using the notion of two--scale convergence, cf. \cite{A92, Ng89}.
\\
By Proposition \ref{compact} we have $u_\e \rightharpoonup u$ weakly in $W^{1,p}(\Omega; \R^3)$, so in particular weakly in $W^{1,2}(\Omega; \R^3)$ and $m_\e \circ w_\e \rightharpoonup m$ weakly in $W^{1,q}(\Omega; \R^3)$ with $m \in W^{1,2q}(\Omega; \s^2)$ and where $q := \frac{2ps}{p+s(p+2)} \in (1, \tfrac{2p}{p+2})$ with $s$ the exponent in \eqref{5.15a}. Note that $1 < q < 2$ and in particular  $m \in W^{1,2}(\Omega; \s^2)$. Thanks to \cite[Proposition 1.14]{A92}, see also \cite[Theorem 9.9]{cioranescu-donato}, and to \cite[Theorem 9]{LNW02}, we get, up to subsequences,
\begin{alignat}{3}
    u_\e &\wts u && \quad \textnormal{ weakly two-scale in } L^2(\Omega; \R^3), 
    \nonumber \\
    \nabla u_\e & \wts \nabla u + \nabla_y u_1(x, y) && \quad \textnormal{ weakly two-scale in } L^2(\Omega \times Y; \R^{3 \times 3})
    \label{2snablaue} 
\end{alignat}
for some $u_1 \in L^2(\Omega; H^1_\per(Y; \R^3))$. For every $\e > 0$, we define the set
\begin{equation*}
    E_\e := \left\{ x \in \Omega \ : \ |\nabla u_\e(x)| \le \e^{-\frac{\alpha}{2}} \right\}.
\end{equation*}
Due to \eqref{2snablaue} and the weak convergence $u_\e \rightharpoonup u$ in $W^{1,2}(\Omega; \R^3)$, granting us uniform boundedness of the gradients, we can apply Lemma \ref{lmm-neuk} obtaining
\begin{equation}
\label{2scalechigrad}
    \chi_{E_\e} \nabla u_\e \wts \nabla u + \nabla_y u_1(x, y) \quad \textnormal{ weakly two-scale in } L^2(\Omega \times Y; \R^{3 \times 3}).
\end{equation}
Thanks to the non-negativity of the energy density $W$, along with the fact that $W$ has a minimizer at $\Id$ by \ref{H4}, we get
\begin{equation}
\label{est-1-lb}
\begin{aligned}
    \frac{1}{\e^{2\alpha}} &\int_\Omega W\left(\frac{x}{\e}, \Id + \e^\alpha \nabla u_\e, m_\e(x + \e^\alpha u_\e(x)) \right) \dx \\
    \quad \ge \frac{1}{\e^{2\alpha}} &\int_\Omega W \left(\frac{x}{\e}, \Id + \e^\alpha \chi_{E_\e} \nabla u_\e, m_\e(x + \e^\alpha u_\e(x)) \right) \dx.
\end{aligned}
\end{equation}
Recall that by \ref{H5} we have
\begin{equation}
\label{newH5}
    \left|W(y, \Id + F, \nu) - \frac{1}{2} Q(y, F, \nu)\right| \le \tilde{r}(|F|)|F|^2,
\end{equation}
for any $F \in \R^{3 \times 3}$, for any $\nu \in \s^2$ and for a.e. $y \in \R^3$, for a continuous function $\tilde{r}$. Furthermore, we have that 
\begin{equation*}
    \| \e^\alpha \chi_{E_\e} \nabla u_\e \|_{L^\infty(\Omega; \R^{3 \times 3})} \le \e^\frac{\alpha}{2}.
\end{equation*}
This, along with \eqref{newH5}, implies that for small enough $\e$ we can approximate $W$ by the quadratic form $Q$, in particular
\begin{equation}
\label{quadratic+reima}
\begin{aligned}
    & \ \frac{1}{\e^{2\alpha}} \int_\Omega W\left(\frac{x}{\e}, \Id + \e^\alpha \chi_{E_\e} \nabla u_\e, m_\e(x + \e^\alpha u_\e(x)) \right) \dx \\
    \ge & \ \frac{1}{2\e^{2\alpha}} \int_\Omega Q \left(\frac{x}{\e}, \e^\alpha \chi_{E_\e} \nabla u_\e, m_\e(x + \e^\alpha u_\e(x)) \right) \dx - r(\e^\frac{\alpha}{2}) \|\chi_{E_\e} \nabla u_\e \|^2_{L^2(\Omega; \R^{3 \times 3})} \\
    = & \ \frac{1}{2} \int_\Omega Q \left(\frac{x}{\e}, \chi_{E_\e} \nabla u_\e, m_\e(x + \e^\alpha u_\e(x)) \right) \dx - r(\e^\frac{\alpha}{2}) \|\chi_{E_\e} \nabla u_\e \|^2_{L^2(\Omega; \R^{3 \times 3})}.
\end{aligned}
\end{equation}
The second term on the right-hand side of \eqref{quadratic+reima} goes to $0$ as $\e \to 0$ due to the boundedness of $\|\chi_{E_\e} \nabla u_\e \|^2_{L^2(\Omega; \R^{3 \times 3})}$. Recall that quadratic forms are lower semicontinuous with respect to the two-scale convergence, see \cite[Proposition 3.2.2]{N10} for example. Therefore, thanks to \eqref{est-1-lb} and \eqref{quadratic+reima}, and in view of \eqref{2scalechigrad} and the strong convergence of $m_\e \circ w_\e$, see Proposition \ref{compact}, we deduce that
\begin{equation*}
\begin{aligned}
    & \ \liminf_{\e \to 0} \frac{1}{\e^{2\alpha}} \int_\Omega W \left(\frac{x}{\e}, \Id + \e^\alpha \nabla u_\e, m_\e(x + \e^\alpha u_\e(x)) \right) \dx \\
    \ge & \ \liminf_{\e \to 0} \frac{1}{2} \int_\Omega Q \left(\frac{x}{\e}, \chi_{E_\e} \nabla u_\e, m_\e(x + \e^\alpha u_\e(x)) \right) \dx \\
    \ge & \ \frac{1}{2} \iint_{\Omega \times Y} Q \left(y, \nabla u(x) + \nabla_y u_1(x,y), m(x) \right) \dx \dif y \\
    \ge & \ \frac{1}{2} \int_\Omega \inf \left\{ \int_Y Q \left(y, \nabla u(x) + \nabla_y \phi(y), m(x) \right) \dif y \ : \ \phi\in H^1_\per (Y; \R^3) \right\} \dx \\
    = & \ \frac{1}{2} \int_\Omega Q_\hom (\nabla u(x), m(x)) \dx,
\end{aligned}
\end{equation*}
as desired.

\smallskip
{\bf{Step 2: Lower bound for the exchange energy term.}}
Regarding the exchange term of the energy functional, we show that
\begin{equation*}
    \liminf_{\e \to 0} \int_\Omega a\left(\frac{x}{\e}\right) |\nabla m_\e(x + \e^\alpha u_\e(x))|^2 |\det \nabla w_\e| \dx \ge \int_\Omega T_\hom(m(x), \nabla m(x)) \dx,
\end{equation*}
with $T_\hom$ being defined as in \eqref{Thom}. To this end, we adapt the strategy used in \cite[Theorem 2.2]{DDF20}, again using the notion of two--scale convergence.

Chosen $\psi\in C^\infty(\Omega; H^1_\per(Y; \R^3))$, we deduce that
\begin{equation*}
\begin{aligned}
    &\int_\Omega a\left(\frac{x}{\e}\right) |\nabla m_\e \circ w_\e|^2 |\det \nabla w_\e| \dx \\
    \ge 2 &\int_\Omega \int_Y a\left(\frac{x}{\e}\right) \left[\nabla m_\e \circ w_\e (\det\nabla w_\e)^\frac{1}{2} \right] : \left[\nabla m(x) + \nabla_y \psi(x,y) \right] \dx\dif y \\
    - &\int_\Omega \int_Y a\left(\frac{x}{\e}\right) |\nabla m(x) + \nabla_y \psi(x,y)|^2 \dx\dif y.
\end{aligned}
\end{equation*}
Here, we have used the fact that the difference of the integrand on the left-hand side and the one on the right-hand side is a perfect square. Now, thanks to Proposition \ref{prop:two_scale} we can pass to the two-scale limit, yielding
\begin{equation*}
\begin{aligned}
    \liminf_{\e \to 0} &\int_\Omega a\left(\frac{x}{\e}\right) |\nabla m_\e \circ w_\e|^2 |\det \nabla w_\e| \dx \\
    \ge 2 &\int_\Omega \int_Y a(y) [\nabla m(x) +  \nabla_y\phi(x,y)] : [\nabla m(x) +  \nabla_y \psi(x,y)] \dx\dif y \\
    - &\int_\Omega \int_Y a(y) |\nabla m(x) +  \nabla_y \psi(x,y)|^2 \dx\dif y.
\end{aligned}
\end{equation*}
By density, there exists a sequence $(\psi_n)_n \subset C^\infty(\Omega; H^1_\per(Y; \R^3))$ such that $\psi_n \to \phi$ strongly in $L^2(\Omega; H^1_\per(Y; \R^3))$, where $\phi$ as in \eqref{5.10}. Hence, passing to the limit as $n \to +\infty$, we get
\begin{equation*}
\begin{aligned}
    \liminf_{\e \to 0} &\int_\Omega a\left(\frac{x}{\e}\right) |\nabla m_\e \circ w_\e|^2 |\det \nabla w_\e| \dx \\
    \ge &\int_\Omega \int_Y a(y)|\nabla m(x) +  \nabla_y\phi(x,y)|^2 \dx\dif y \\
    \ge &\int_\Omega \inf\left\{\int_Y a(y) |\nabla m(x) + \Xi(y)|^2 \dif y \ : \ \Xi \in  H^1_\per(Y; \R^3) \right\} \dx,
\end{aligned}
\end{equation*}
concluding the proof.


\subsection{Recovery sequences and compactness}
\label{conv_rec_seq}

In this section, we are going to construct a recovery sequence and to prove the necessary convergences needed to prove the optimality of the upper bound. To this end, let $u \in W^{1,p}_g(\Omega; \R^3)\cap  W^{2,\infty}(\Omega;\R^3)$ and $m \in H^1(\Omega; \s^2)$. Moreover, partly following the strategy of \cite[Theorem 2.3]{DDF20}, let us define $\pi$ as the pointwise projection operator from $U_\delta$ to $\s^2$, where $U_\delta$ is the tubular neighborhood of size $\delta$ of the sphere $\s^2$. Note that it is possible to choose $\delta$ small enough so that $\pi \in C^1(\overline{U}_\delta; \s^2)$. For any $\e > 0$ and a.e. $x \in \Omega$, let us consider the sequences 
\begin{equation}
\label{def:recoveryseq}
\begin{aligned}
    u_\e(x) &:= u(x) + \e \, \psi \left( \frac{x}{\e} \right), \\
    w_\e(x) &:= x + \e^\alpha u_\e(x), \\
    \widehat{m}_\e(w_\e(x)) & := m(x) + \e \, \varphi \left(x, \frac{x}{\e}\right) \\
    m_\e(w_\e(x)) & := \pi[\widehat{m}_\e(w_\e(x))],
\end{aligned}
\end{equation}
where $\psi \in C^\infty_\per(Y; \R^3)$ and $\varphi \in C^\infty(\overline{\Omega}, W^{1,\infty}_\per(Y; \R^3))$. We want to prove that
\begin{alignat}{3}
    w_\e &\to id &&\qquad \textnormal{strongly in } W^{1,p}(\Omega; \R^3),
    \nonumber \\
    u_\e &\rightharpoonup u &&\qquad \textnormal{weakly in } W^{1,p}(\Omega; \R^3),
    \nonumber \\
    \det\nabla w_\e &\to 1 &&\qquad \textnormal{strongly in } L^\frac{p}{3}(\Omega),
    \label{rs_conv_det} \\
    |\det\nabla w_\e|^{-1} &\to 1 &&\qquad \textnormal{strongly in } L^1(\Omega),
    \label{rs_conv_det_inv} \\   
    m_\e \circ w_\e &\rightharpoonup m  &&\qquad \textnormal{weakly in } W^{1,q}(\Omega; \R^3),
    \label{rs_conv_mw} \\
    m_\e \circ w_\e &\to m  &&\qquad \textnormal{strongly in } L^r(\Omega; \R^3) \quad \textnormal{for any } r \in [1,+\infty),
    \label{rs_conv_mw_strong} \\
    \chi_{w_\e(\Omega)} m_\e &\to \chi_\Omega m &&\qquad \textnormal{strongly in } L^r(\R^3; \R^3) \quad \textnormal{for any } r \in [1,+\infty),
    \label{rs_conv_chim} \\
    \chi_{w_\e(\Omega)} \nabla m_\e &\rightharpoonup \chi_\Omega \nabla m &&\qquad \textnormal{weakly in } L^2(\R^3; \R^{3 \times 3}),
    \label{rs_conv_chi_nablam}
\end{alignat}
where $q := \frac{2ps}{p+s(p+2)} \in (1, \tfrac{2p}{p+2})$, and where $s$ is the exponent in \eqref{5.15a}. Moreover, we need to prove that
\begin{equation}
\label{2sc_limsup}
    \nabla m_\e \circ w_\e |\det\nabla w_\e|^\frac{1}{2} \wts \nabla m(x) + \nabla_y \phi(x,y) \quad \textnormal{ weakly two-scale in } L^2(\Omega \times Y; \R^3),
\end{equation}
where $\phi \in L^2(\Omega; H^1_\per(Y; \R^3))$ is the minimum of the homogenized problem in \eqref{Thom} when integrated over $\Omega$. Note that by Lax-Milgram’s Lemma, such a minimum $\phi$ exists and is unique. It is easy to see that $(u_\e)_\e \subset W^{1,p}_g(\Omega;\R^3)$, and so also $(w_\e)_\e \subset W^{1,p}(\Omega;\R^3)$, implying the first two convergences (in particular with the second one being strong as well).

Regarding the determinants, proceeding as in Step 2 of the proof of Proposition \ref{compact}, we get
\begin{equation*}
\begin{aligned}
    \int_\Omega |\det(\nabla w_\e)-1|^\frac{p}{3} \dx \le \int_\Omega |\e^\alpha\nabla u_\e|^p \dx + 3 \int_\Omega |\e^\alpha\nabla u_\e|^\frac{p}{3} \dx + 6 \int_\Omega |\e^\alpha\nabla u_\e|^\frac{2p}{3}\dx.
\end{aligned}
\end{equation*}
Since we know already that $(u_\e)_\e \subset W^{1,p}_g(\Omega;\R^3)$, then taking the limit as $\e \to 0$, we get convergence \eqref{rs_conv_det}.

Note that we need to show that $w_\e \in \mathcal{Y}$, see \eqref{eq:class-Y}. Thanks to Theorem \ref{thm:injective} we get the injectivity of $w_\e$ and the sign condition $\det \nabla w_\e > 0$ thanks to the fact that $\e \|\nabla u \|_\infty \le c(\Omega)$. Following the lines of \cite[Section 4]{AKM24}, it also follows that $w_\e$ is bi-Lipschitz for small enough $\e$. In particular, 
\begin{equation}
\label{eq:infdet}
    \inf_{x\in\Omega} \det \nabla w_\e := l_\e > l > 0.
\end{equation}
The fact that $w_\e$ is bi-Lipschitz holds that \eqref{5.15a} is satisfied. Moreover, setting $B_\e^\delta$ as in Step 3 of the proof of Proposition \ref{compact}, namely \eqref{B_eps_d}, we have that
\begin{equation*}
    \int_{B_\e^\delta} |(\det\nabla w_\e)^{-1}-1| \dx = \int_{B_\e^\delta} \frac{|1-\det\nabla w_\e|}{|\det\nabla w_\e|} \dx \le C |B_\e^\delta| \to 0.
\end{equation*}
Note that due to the convergence of $\det\nabla w_\e$ as well as \eqref{eq:infdet}, the constant $C$ does not depend on $\e$. Therefore, convergence \eqref{rs_conv_det_inv} follows as in Step 3 of of the proof of Proposition \ref{compact}.

By the definition of the sequences in \eqref{def:recoveryseq}, it follows immediately that
\begin{alignat}{3}
    m_\e(w_\e) &\to m = \pi[m] &&\qquad \textnormal { strongly in } L^2(\Omega; \R^3),
    \label{eq:strongconhatme} \\
    \nabla \widehat{m}_\e(w_\e) &\sts \nabla m + \nabla_y \varphi(\cdot, y) &&\qquad \textnormal{ strongly two-scale in } L^2(\Omega\times Q; \R^{3 \times 3}).
    \nonumber
\end{alignat} 
Thanks to the first convergence and the fact that $|m_\e \circ w_\e| = 1$ for any $\e > 0$ and a.e. $x \in \Omega$, see \eqref{def:recoveryseq}, we get \eqref{rs_conv_mw_strong} exploiting Lemma \ref{l5.5}. As for convergence \eqref{rs_conv_chim} and \eqref{rs_conv_chi_nablam}, we can proceed as in Steps 6 and 7 of the proof of Proposition \ref{compact}.
\\
Now, bearing in mind the regularity of $\varphi$, it follows that for $\e$ small enough $\widehat{m}_\e \circ w_\e(x) \in U_\delta$ for a.e. $x \in \Omega$. Moreover, the regularity of the projection $\pi$ implies that there exists a constant $C > 0$ such that 
\begin{equation*}
    |\nabla m_\e(w_\e)| \le C |\nabla \widehat{m}_\e(w_\e)| \qquad \textnormal{ a.e. in } \Omega.
\end{equation*}
In view of \eqref{eq:strongconhatme} as well as the boundedness of $(\nabla \widehat{m}_\e(w_\e))_\e$ in $L^2(\Omega; \R^{3 \times 3})$, the following convergences, up to a subsequence, hold:
\begin{alignat}{3}
    \widehat{m}_\e(w_\e) &\rightharpoonup m &&\qquad \textnormal{ weakly in } H^1(\Omega; \s^2),
    \label{eq:weakconvhatwe} \\
    m_\e(w_\e) &\rightharpoonup m &&\qquad \textnormal{ weakly in } H^1(\Omega; \s^2).
    \nonumber
\end{alignat}
Note that the second convergence implies directly \eqref{rs_conv_mw}, since $q < 2$. 
Moreover, it is possible to prove that
\begin{equation}
\label{conv_con_par}
    \nabla m_\e(w_\e) (\det(\nabla w_\e))^\frac{1}{2} \wts (\nabla m + \nabla_y \varphi) \nabla \pi[m] \quad \textnormal{ weakly two-scale in } L^2(\Omega \times Y; \R^{3 \times 3}).
\end{equation}
Indeed, a direct computation shows that for a.e. $x \in \Omega$,
\begin{equation*}
\begin{aligned}
    &\nabla m_\e(w_\e(x)) (\det(\nabla w_\e))^\frac{1}{2} \\
    = &\left[\nabla m(x) + \e \nabla_x \varphi \left(x, \frac{x}{\e} \right) + \nabla_y \varphi \left(x, \frac{x}{\e} \right) \right] \nabla \pi[\widehat{m}_\e(w_\e(x))] (\det(\nabla w_\e))^\frac{1}{2}.
\end{aligned}
\end{equation*}
Deploying the regularity of $\varphi$, it easily follows that we have the following convergences
\begin{alignat*}{3}
    \e \nabla_x \varphi \left(x, \frac{x}{\e}\right) &\to 0 &&\quad \textnormal{ strongly in } L^\infty(\Omega; \R^{3 \times 3}), \\
    \nabla_y \varphi\left(x, \frac{x}{\e} \right) &\sts \nabla_y \varphi &&\quad \textnormal{ strongly two-scale in } L^2(\Omega \times Y; \R^{3 \times 3}).
\end{alignat*}
Now, we show that 
\begin{equation}
\label{conv_aux}
    \nabla \pi[\widehat{m}_\e(w_\e(x))] (\det(\nabla w_\e))^\frac{1}{2} \to \nabla\pi[m] \quad \textnormal{ strongly in } L^2(\Omega; \R^{3 \times 3}).
\end{equation}
Indeed, 
\begin{equation}
\label{eq:convergenceofgrad}
\begin{aligned}
    & \ \| \nabla \pi[\widehat{m}_\e(w_\e(x))] (\det(\nabla w_\e))^\frac{1}{2} - \nabla\pi[m] \|_{L^2(\Omega; \R^{3 \times 3})} \\
    \le & \ \| (\nabla \pi[\widehat{m}_\e(w_\e(x))] - \nabla\pi[m]) (\det(\nabla w_\e))^\frac{1}{2}  \|_{L^2(\Omega; \R^{3 \times 3})} \\
    & + \|\nabla\pi[m] ((\det(\nabla w_\e))^\frac{1}{2}-1 )\|_{L^2(\Omega; \R^{3 \times 3})}.
\end{aligned}
\end{equation}
The first term in \eqref{eq:convergenceofgrad} converges to zero, since
\begin{equation}
\label{eq:1term}
\begin{aligned}
    &\| (\nabla \pi[\widehat{m}_\e(w_\e(x))]-\nabla\pi[m]) (\det(\nabla w_\e))^\frac{1}{2}  \|_{L^2(\Omega; \R^{3 \times 3})} \\
    = & \ \int_\Omega \left|\nabla \pi[\widehat{m}_\e(w_\e(x))]-\nabla\pi[m] \right|^2 \det(\nabla w_\e) \dx \\
    = & \ \int_\Omega \left|\nabla \pi[\widehat{m}_\e(w_\e(x))]-\nabla\pi[m] \right|^2 \left(\det(\nabla w_\e)-1\right) \dx \\ 
    & + \int_\Omega \left|\nabla \pi[\widehat{m}_\e(w_\e(x))]-\nabla\pi[m] \right|^2 \dx. 
\end{aligned}
\end{equation}
Since $\nabla\pi [\widehat{m}_\e(w_\e)]$ strongly converges in $L^2(\Omega, \R^{3 \times 3})$ to  $\nabla\pi[m]$, the second integral in \eqref{eq:1term} strongly converges to $0$ in $L^2(\Omega, \R^{3 \times 3})$. Moreover, due to the strongly convergence of $\det(\nabla w_\e)$ in $L^1(\Omega)$, see \eqref{rs_conv_det}, together with the fact that the regularity of $\mathbb{S}^2$ provides an $L^\infty$-bound on $\nabla \pi[m]$, the first term in \eqref{eq:1term} strongly converges to $0$ in $L^2(\Omega, \R^{3 \times 3})$ as well.
\\
As for the second term in \eqref{eq:convergenceofgrad}, first note that using the subadditivity of the function $t\mapsto \sqrt{t}$, for $t\ge 0$, it follows that
\begin{equation*}
    (\det(\nabla w_\e))^\frac{1}{2}-1 \le (\det(\nabla w_\e)-1)^\frac{1}{2}.
\end{equation*}
This, combined with the $L^\infty$-bound on  $\nabla \pi[m]$, yields to conclude that also the second term in \eqref{eq:convergenceofgrad} strongly converges to $0$ in $L^2(\Omega, \R^{3 \times 3})$. Indeed,
\begin{equation*}
\begin{aligned}
    \|\nabla\pi[m] ((\det(\nabla w_\e))^\frac{1}{2}-1) \|_{L^2(\Omega; \R^{3 \times 3})} 
    &= \int_\Omega |\nabla\pi[m]|^2 ((\det(\nabla w_\e))^\frac{1}{2}-1)^2 \dx \\
    &\le C \int_\Omega (\det(\nabla w_\e)-1) \dx,
\end{aligned}
\end{equation*}
where the right-hand side converges to $0$ in $L^1(\Omega)$, see \eqref{rs_conv_det}. This proves \eqref{conv_aux}, which in turn implies \eqref{conv_con_par}.


Noticing that $\pi[m] = m$ in $\Omega$, we get that $\nabla m \nabla \pi[m]= \nabla m$ for a.e. $x \in \Omega$. Hence, \eqref{conv_con_par} is equivalent to
\begin{equation*}
    \nabla m_\e(w_\e)(\det(\nabla w_\e))^\frac{1}{2} \wts \nabla m + \nabla_y \varphi\nabla \pi[m] \quad \textnormal{ weakly two-scale in } L^2(\Omega \times Y; \R^{3 \times 3}).
\end{equation*}
We prove that the previous convergence is actually strong. Indeed, note that
\begin{equation*}
\begin{aligned}
    & \ \left| \|\nabla m_\e(w_\e)(\det(\nabla w_\e))^\frac{1}{2}\|_{L^2(\Omega; \R^{3 \times 3})} - \|\nabla m + \nabla_y \varphi\nabla \pi[m] \|_{L^2(\Omega \times Y; \R^{3 \times 3})} \right| \\
    \le & \ \|\nabla m_\e(w_\e)(\det(\nabla w_\e))^\frac{1}{2} - (\nabla m + \nabla_y \varphi\nabla \pi[m]) \|_{L^2(\Omega \times Y; \R^{3 \times 3})}
\end{aligned}
\end{equation*}
and that
\begin{equation}
\label{eq_aux:difference}
\begin{aligned}
    & \ \nabla m_\e(w_\e)(\det(\nabla w_\e))^\frac{1}{2} - (\nabla m + \nabla_y \varphi\nabla \pi[m]) \\
    = & \ \nabla \widehat{m}_\e(w_\e)\nabla \pi[\widehat{m}_\e(w_\e)](\det(\nabla w_\e))^\frac{1}{2} - (\nabla m + \nabla_y \varphi)\nabla \pi[m]\\
    = & \ (\nabla\widehat{m}_\e(w_\e) - (\nabla m + \nabla_y\varphi)) \nabla \pi[\widehat{m}_\e(w_\e)] (\det(\nabla w_\e))^\frac{1}{2}\\
    &\quad +(\nabla m + \nabla_y \varphi)(\nabla \pi[\widehat{m}_\e(w_\e)](\det(\nabla w_\e))^\frac{1}{2} - \nabla \pi[m]).
\end{aligned}
\end{equation}
The first term in the right-hand side of \eqref{eq_aux:difference} converges to zero in $L^2(\Omega; \R^{3 \times 3})$. Indeed, due to the regularity of $\s^2$, we get an $L^\infty$-bound on $\nabla \pi$. Furthermore, owing \eqref{def:recoveryseq}, we have that $\nabla \widehat{m}_\e(w_\e) - (\nabla m + \nabla_y \varphi) = \e \nabla_x \varphi(x, \tfrac{x}{\e})$. Hence, we conclude by the Lebesgue's dominated convergence theorem. Likewise, we have that $(\nabla m + \nabla_y \varphi)(\nabla \pi[\widehat{m}_\e(w_\e)](\det(\nabla w_\e))^\frac{1}{2} - \nabla \pi[m]) \to 0$ strongly in $L^2(\Omega; \R^{3 \times 3})$. Therefore, we conclude that 
\begin{equation}
\label{eq_aux:strong_conv}
    \nabla m_\e(w_\e)(\det(\nabla w_\e))^\frac{1}{2}  \sts \nabla m + \nabla_y \varphi\nabla \pi[m] \quad \textnormal{ strongly two-scale in } L^2(\Omega \times Y; \R^{3 \times 3}).
\end{equation}
Thanks to \eqref{eq_aux:strong_conv} and to the regularity of the function $\varphi$, see \eqref{def:recoveryseq}, we have that
\begin{equation}
\label{conv_exch_1}
\begin{aligned}
    \lim_{\e \to 0} &\int_\Omega a\left(\frac{x}{\e}\right) |\nabla m_\e(w_\e)|^2 |\det\nabla w_\e| \dx \\
    = &\iint_{\Omega \times Y} a(y) |\nabla m(x) + \nabla_y \varphi(x, y) \nabla \pi[m(x)]|^2 \dx \dif y.
\end{aligned}
\end{equation}
Now let us define
\begin{equation*}
    m_1(x,y) := \phi[m(x), \nabla m(x)](y),
\end{equation*}
where $\phi$ is the minimum of the homogenized problem in \eqref{Thom} when integrated over $\Omega$. By density, there exists a sequence $(\varphi_k)_k \subset \mathcal{D}(\Omega; W^{1,\infty}_\per(Y; \R^3))$ such that
\begin{equation*}
    \varphi_k \to m_1 \quad \textnormal{ strongly in } L^2(\Omega; H^1_\per(Q; \R^3)).
\end{equation*}
Note that this yields
\begin{equation}
\label{opt_rec_exchange}
    \lim_{k\to\infty} \int_{\Omega\times Y} a(y) |\nabla m + \nabla_y \varphi_k(x, y)\nabla \pi[m]|^2 \dx = \int_{\Omega \times Y} a(y) |\nabla m + \nabla m_1|^2 \dx.
\end{equation}
This conclude the proof, since the limiting functional is equal to the right-hand side of the above equality and this entails \eqref{2sc_limsup}.


\subsection{Recovery of the limiting energy}
\label{rec_seq}

In this section, we are finally going to prove the optimality of the upper bound, showing that the sequences constructed in Subsection \ref{conv_rec_seq}, in particular as in \eqref{def:recoveryseq}, allow us to recover the upper bound needed in order to finish the proof of Theorem \ref{main_thm}.

\begin{proof}[Proof of Theorem \ref{main_thm}]
    Given our rescaled magnetoelastic energy $F_\e$ defined as in \eqref{resc-en}, we have proven already in Subsection \ref{low-bou} the $\Gamma$-liminf inequality with respect to the convergence in Propositions \ref{compact} and \ref{prop:two_scale} with respect to the homogenized functional $F_{\hom}$ defined by \eqref{hommagnetoel}. In order to deal with the upper bound for our family of energy functionals, let us firstly remark, by the subadditivity of the $\limsup$, that
    \begin{equation*}
    \begin{aligned}
        \limsup_{\e \to 0} F_\e (u_\e, m_\e)
        \le & \, \limsup_{\e \to 0} \frac{1}{\e^{2\alpha}} \int_\Omega W \left(\frac{x}{\e}, \Id + \e^\alpha \nabla u_\e(x), m_\e(x + \e^\alpha u(x)) \right) \dx \\
        + & \, \limsup_{\e \to 0} \int_\Omega a \left(\frac{x}{\e}\right) |\nabla m_\e(x + \e^\alpha u_\e(x))|^2 |\det\nabla w_\e| \dx \\
        + & \, \limsup_{\e \to 0} \, \frac{\mu_0}{2} \int_{\R^3} |\nabla \psi_{m_\e}|^2 \dx.
    \end{aligned}
    \end{equation*}
    In Subsection \ref{lim_magnetostatic} we dealt with the magnetostatic self-energy term of the energy, proving its limit. Moreover, in Subsection \ref{conv_rec_seq} we constructed the recovery sequence of admissible states, see \eqref{adm}, converging as in Propositions \ref{compact} and \ref{prop:two_scale} and we proved the optimality of the upper bound for the exchange energy term, see \eqref{opt_rec_exchange}. The last step in order to complete the proof is to deal with the elastic term of the energy.
    \\
    Let us consider the sequences constructed in \eqref{def:recoveryseq}. Keeping in mind the definition of $Q_\hom$, namely \eqref{Qhom}, and the density of the space $C^\infty_\per(\Omega; \R^3)$ in $H^1_\per(\Omega; \R^3)$, we know that for any fixed $\eta > 0$, there exists $\psi\in C^\infty_\per(\Omega; \R^3)$ such that for any $\nu \in \s^2$ we have
    \begin{equation}
    \label{almostminforelasticpart}
        \int_Y Q(y, \nabla u(x) + \nabla_y\psi(y), \nu) \dif y \le Q_\hom(\nabla u, \nu) +\eta.
    \end{equation}
    Employing \eqref{newH5}, we deduce that 
    \begin{equation*}
    \begin{aligned}
        & \limsup_{\e \to 0} \frac{1}{\e^{2\alpha}} \int_\Omega W \left( \frac{x}{\e}, \Id + \e^\alpha \left(\nabla u + \nabla_y \psi \left(\frac{x}{\e}\right)\right), m_\e(w_\e(x)) \right) \dx \\
        \le \, & \limsup_{\e \to 0} \frac{1}{2\e^{2\alpha}} \int_\Omega Q \left(\frac{x}{\e}, \e^\alpha \left(\nabla u(x) + \nabla_y \psi \left(\frac{x}{\e}\right) \right), m_\e(w_\e(x)) \right) \dx \\
        & + \limsup_{\e \to 0} \ r(\e^\alpha) \int_\Omega \left|\nabla u(x) + \nabla_y \psi\left(\frac{x}{\e}\right)\right|^2 \dx \\
        \le \, & \limsup_{\e \to 0} \frac{1}{2}\int_\Omega Q \left(\frac{x}{\e}, \nabla u(x) + \nabla_y \psi\left(\frac{x}{\e}\right), m_\e(w_\e(x)) \right) \dx \\
        = \, & \frac{1}{2} \int_\Omega \int_Y Q(y, \nabla u(x) + \nabla_y\psi(y), m(x)) \dx,
    \end{aligned}
    \end{equation*}
    where in the last equality we used the Riemann-Lebesgue Lemma and the continuity of $Q$ with respect to the third variable alongside \eqref{rs_conv_chim}. Finally, using \eqref{almostminforelasticpart}, it follows that 
    \begin{equation*}
    \begin{aligned}
        \limsup_{\e\to 0} & \int_\Omega W \left(\frac{x}{\e}, \Id + \e^\alpha \left(\nabla u + \nabla_y \psi\left(\frac{x}{\e}\right)\right),    m_\e(w_\e(x)) \right) \dx \\
        \le \, & \frac{1}{2} \int_\Omega \int_Y Q(y, \nabla u(x) + \nabla_y\psi(y), m(x)) \dx \\
        \le \, & \frac{1}{2} \int_\Omega Q_\hom (\nabla u(x), m(x)) \dx + \frac{\eta}{2}|\Omega|,
    \end{aligned}
    \end{equation*}
    and by the arbitrariness of $\eta$ we get the desired upper bound for the elastic part. This concludes the proof.
\end{proof}


\section{On the commutability of homogenization and linearization}
\label{sec:commute}

As already highlighted in the introduction, a direct consequence of Theorem \ref{main_thm} is that by performing simultaneous homogenization and linearization, we identify the same limiting model as by linearizing first and homogenizing afterward. This seems to suggest that homogenization and linearization processes in this setting may be decoupled. We devote the last section of the paper to describe the iterated linearization and homogenization analysis in a bit more details. We begin by stating a linearization result for the energies $(G_{\e})_{\e}$ in \eqref{eq:en-hom}. To this end, for $\e > 0$ and $\delta>0$ we first introduce the class of admissible deformations with scaled displacements:
\begin{equation*}
    \A_\delta := \{ (u, m) \in W^{1,p}_g(\Omega; \R^3) \times H^1((id + \delta u)(\Omega); \s^2) \ : \ id + \delta u \in \Y \},
\end{equation*}
cf. the definition of $\A_\e$ in \eqref{adm}, as well as the rescaled energies
\begin{equation}
\label{resc-en/delta}
\begin{aligned}
    F_{\e}^{\delta}(u,m)
    := \ &\frac{1}{\delta^{2}} \int_\Omega W \left(\frac{x}{\e}, \Id + \delta \nabla u(x), m(x + \delta u(x)) \right) \dx \\
    + &\int_\Omega a \left(\frac{x}{\e}\right) |\nabla m(x + \delta u(x))|^2 |\det(\Id+\delta\nabla u)| \dx + \frac{\mu_0}{2} \int_{\R^3} |\nabla \psi_m|^2 \dx,
\end{aligned}
\end{equation}
for all $(u,m)\in \mathcal{A}_{\delta}$. The limiting linearized energy as $\delta\to 0$ is given by the functional
\begin{equation}
\label{lin-en/delta}
    G_{\e}^{\textrm{lin}}(u,m) := \int_\Omega Q \left(\frac{x}{\e}, \nabla u(x), m(x) \right) \dx + \int_\Omega a \left(\frac{x}{\e}\right) |\nabla m(x)|^2 \dx + \frac{\mu_0}{2} \int_{\R^3} |\nabla \psi_m|^2 \dx,
\end{equation}
for all $(u,m)\in W^{1,p}_g(\Omega; \R^3) \times H^1(\Omega; \mathbb{S}^2)$. The compactness result in this setting is a direct consequence of Proposition \ref{compact}. We report its statement  for the sake of completeness.

\begin{prop}
\label{compact/delta}
     Assume \ref{H1}--\ref{H5}. Let $(u_\delta, m_\delta) \in \A_\delta$ be such that 
    \begin{equation*}
        \sup_{\delta > 0} F_\e^{\delta}(u_\delta, m_\delta) < + \infty,
    \end{equation*}
    where the energy $F_\e^{\delta}$ is defined as in \eqref{resc-en/delta}. Then, there exists $(u,m) \in W^{1,p}_g(\Omega; \R^3) \times H^1(\Omega; \mathbb{S}^2)$ such that, up to subsequences, setting $w_\delta := id + \delta u_\delta$, there holds
    \begin{alignat*}{3}
        w_\delta &\to id &&\qquad \textnormal{strongly in } W^{1,p}(\Omega; \R^3), \\
        u_\delta &\rightharpoonup u &&\qquad \textnormal{weakly in } W^{1,2}(\Omega; \R^3), \\
        \det\nabla w_\delta &\to 1 &&\qquad \textnormal{strongly in } L^\frac{p}{3}(\Omega), \\ 
        |\det\nabla w_\delta|^{-1} &\to 1 &&\qquad \textnormal{strongly in } L^1(\Omega), \\
        m_\delta \circ w_\delta &\rightharpoonup m  &&\qquad \textnormal{weakly in } W^{1,q}(\Omega; \R^3), \\
        m_\delta \circ w_\delta &\to m  &&\qquad \textnormal{strongly in } L^r(\Omega; \R^3) \quad \textnormal{for any } r \in [1,\infty), \\
        \chi_{w_\e(\Omega)} m_\delta &\to \chi_\Omega m &&\qquad \textnormal{strongly in } L^r(\R^3; \R^3) \quad \textnormal{for any } r \in [1,\infty), \\
        \chi_{w_\delta(\Omega)} \nabla m_\delta &\rightharpoonup \chi_\Omega \nabla m &&\qquad \textnormal{weakly in } L^2(\R^3; \R^{3 \times 3}),
    \end{alignat*}
    where $q := \frac{2ps}{p+s(p+2)} \in (1, \tfrac{2p}{p+2})$, and $s$ is the exponent in \eqref{5.15a}.
    \qed
\end{prop}

The linearization result reads then as follows.

\begin{prop}
\label{prop:only-lin}
    Assume \ref{H1}-\ref{H5}. Then, the family of magnetoelastic energies $(F_\e^{\delta})_\delta$, given by \eqref{resc-en/delta}, $\Gamma$-converges with respect to the convergence in Proposition \ref{compact/delta} to the linearized functional $G_\e^{\textrm{lim}}$ defined by \eqref{lin-en/delta}.
\end{prop}
\begin{proof}
    We omit the proof of this proposition, for it follows by a direct adaptation of the arguments in \cite{AKM24}, along the lines of the proof of Theorem \ref{main_thm}.
\end{proof}

The homogenization result of the energies $(G_\e^{\textrm{lin}})_\e$ follows then as a corollary of the analysis in \cite{DDF20}. We state it below (without proof) for convenience of the reader.

\begin{prop}
    Assume \ref{H1}-\ref{H5}. Then, the energies $(G_\e^{\textrm{lin}})_\e$ $\Gamma$-converge with respect to the weak topology in $W^{1,p}(\Omega;\R^3)\times H^1(\Omega;\mathbb{S}^2)$ to the homogenized functional $F_{\hom}$ in \eqref{hommagnetoel}.
    \qed
\end{prop}


\section*{Acknowledgements}
The research of E.D, L. D'E. and S.R. has been funded by the Austrian Science Fund (FWF) through grants \href{https://www.doi.org/10.55776/F65}{10.55776/F65}, \href{https://www.doi.org/10.55776/Y1292}{10.55776/Y1292}, \href{https://www.doi.org/10.55776/P35359}{10.55776/P35359}, \href{https://www.doi.org/10.55776/ESP1887024}{10.55776/ESP1887024}, and \href{https://www.doi.org/10.55776/F100800}{10.55776/F100800}. For open-access purposes, the authors have applied a CC BY public copyright license to any author-accepted manuscript version arising from this submission.



\end{document}